\documentclass{article}
\usepackage{graphicx,verbatim, bussproofs,tikz,float, latexsym}
\usetikzlibrary{arrows}
\usepackage[left=3cm,top=3cm,right=3cm,bottom=2cm]{geometry}
\usepackage{pgfplots,multicol}
\usepackage{amsmath,amsthm}
\usepackage{amssymb}
\usepackage{proof}
\usepackage{cmll}
\usepackage{bussproofs}
\usepackage{authblk}
\DeclareSymbolFont{extraup}{U}{zavm}{m}{n}
\DeclareMathSymbol{\vardiamond}{\mathalpha}{extraup}{87}

\newcommand{\commment}[1]{}
\newcommand{\blhd}{\blacktriangleleft_f}
\newcommand{\brhd}{\blacktriangleright_g}

\newcommand{\K}{\mathsf{K}}

\newcommand{\jty}{J^{\infty}}
\newcommand{\mty}{M^{\infty}}

\newcommand{\diam}{\Diamond}
\newcommand{\nomi}{\mathbf{i}}
\newcommand{\nomj}{\mathbf{j}}
\newcommand{\nomk}{\mathbf{k}}
\newcommand{\cnomm}{\mathbf{m}}
\newcommand{\cnomn}{\mathbf{n}}
\newcommand{\cnomo}{\mathbf{o}}

\newcommand{\bigamp}{\mathop{\mbox{\Large \&}}}
\newcommand{\amp}{\mathop{\&}}

\newcommand{\x}{\mathcal{X}}

\renewcommand{\phi}{\varphi}

\renewcommand{\star}{*}
\newcommand{\ca}{\mathbb{A}^\delta}
\newcommand{\A}{\mathbb{A}}
\newcommand{\B}{\mathbb{B}}
\newcommand{\C}{\mathbb{C}}

\newcommand{\bba}{\mathbb{A}}
\newcommand{\bbas}{\mathbb{A}^{\delta}}
\newcommand{\kbbas}{K(\mathbb{A}^{\delta})}
\newcommand{\obbas}{O(\mathbb{A}^{\delta})}
\newcommand{\jir}{J^{\infty}(\bba^{\delta})}
\newcommand{\mir}{M^{\infty}(\bba^{\delta})}

\newcommand{\Diamondblack}{\vardiamond}

\renewcommand{\epsilon}{\varepsilon}
\newtheorem{theorem}{Theorem}[section]
\newtheorem{lemma}[theorem]{Lemma}
\newtheorem{prop}[theorem]{Proposition}

\newtheorem{cor}[theorem]{Corollary}

\newtheorem{definition}[theorem]{Definition}
\newtheorem{example}[theorem]{Example}

\newtheorem{remark}[theorem]{Remark}

\newtheorem{dfn}[theorem]{Definition}

\newtheorem{lem}[theorem]{Lemma}

\newtheorem{corollary}[theorem]{Corollary}

\title{Sahlqvist theory for impossible worlds}

\author[1,3]{Alessandra Palmigiano\thanks{The research of the first and third author has been made possible by the NWO Vidi grant 016.138.314, by the NWO Aspasia grant 015.008.054, and by a Delft Technology Fellowship awarded in 2013.}}
\author[2]{Sumit Sourabh}
\author[1]{Zhiguang Zhao}
\affil[1]{\small Faculty of Technology, Policy and Management, Delft University of Technology, the Netherlands}
\affil[2]{\small Institute for Informatics, University of Amsterdam, the Netherlands}
\affil[3]{\small Department of Pure and Applied Mathematics, University of Johannesburg, South Africa}
\date{}
\begin{document}

\maketitle

\begin{abstract}
We extend unified correspondence theory to Kripke frames with impossible worlds and their associated \emph{regular} modal logics. These are logics the modal connectives of which are not required to be normal: only the weaker properties of additivity $\Diamond x\vee \Diamond y = \Diamond (x\vee y)$ and multiplicativity $\Box x\wedge \Box y = \Box (x\wedge y)$ are required. Conceptually, it has been argued that their lacking necessitation makes regular modal logics better suited than normal modal logics at the formalization of epistemic and deontic settings. From a technical viewpoint, regularity proves to be very natural and adequate for the treatment of algebraic canonicity J\'onsson-style. Indeed, additivity and multiplicativity turn out to be key to extend J\'onsson's original proof of canonicity to the full Sahlqvist class of certain regular distributive modal logics 
naturally generalizing Distributive Modal Logic \cite{GeNaVe05, ConPal12}. Most interestingly, additivity and multiplicativity are key to J\'onsson-style canonicity also in the original (i.e.\ {\em normal}) DML. Our contributions include: the definition of Sahlqvist inequalities for regular modal logics on a distributive lattice propositional base; the proof of their canonicity following J\'onsson's strategy; the adaptation of the algorithm \textsf{ALBA} to the setting of regular modal logics on two non-classical (distributive lattice and intuitionistic) bases; the proof that the adapted \textsf{ALBA} is guaranteed to succeed on a syntactically defined class which properly includes the Sahlqvist one; finally, the application of the previous results so as to obtain proofs, alternative to Kripke's, of the strong completeness of Lemmon's epistemic logics E2-E5 with respect to elementary classes of Kripke frames with impossible worlds.\\

\noindent {\em Keywords:} regular modal logics, epistemic logic, algebraic Sahlqvist canonicity, algorithmic correspondence, regular inductive formulas and inequalities, Lemmon's logics E2-E5.\\
\noindent {\em Math. Subject Class.} 03B45, 06D50, 06D10, 03G10, 06E15.

\end{abstract}

\section{Introduction}

\paragraph{Impossible worlds.}

The formalization of situations in which logical impossibilities are thinkable and sometimes even believable has been a key topic in modal logic since its onset, and has attracted the interest of various communities of logicians over the years. This specific imperfection of cognitive agency can be directly translated in the language of modal logic by stipulating that, for a given agent $a$, the formula $\Diamond_a\bot$ is not a contradiction, and hence, that the necessitation rule is not admissible. Impossible worlds have been introduced by Kripke in \cite{Kr65} in the context of his relational semantic account of modal logics, as an elegant way to invalidate the necessitation rule while retaining all other axioms and rules of normal modal logic, and hence to provide complete semantics for important non-normal modal logics such as Lemmon's systems E2-E4.

More recently, impossible worlds have been used in close connection with counterfactual reasoning, paraconsistency (e.g.\ to model inconsistent databases, cf.\ \cite{Barwise1997}). The reader is referred to \cite{Nolan2013} for a comprehensive survey on impossible worlds.

The logics E2-E4 mentioned above are prominent examples of {\em regular modal logics}, which are classical modal logics (cf.\ \cite{CH90}) in which the necessitation rule is not valid (equivalently, modal logics that do not contain $\Box\top$ as an axiom) but such that $\Box$ distributes over conjunction. Arguably, their lacking necessitation makes regular modal logics better suited than normal modal logics at the formalization of epistemic and deontic settings. To briefly expand on the type of objections against normality raised in these settings, we mention Lemmon's argument in \cite{Le57}, the same paper in which the systems E2-E5 have been introduced together with other logics. The rule of necessitation is not in the systems since its presence  would generate theorems of the form $\Box \phi $. In the context of the interpretation of the $\Box$-operator as moral obligation or scientific but not logical necessity, Lemmon's systems are in line with the view that \textit{nothing should be a scientific law or a moral obligation as a matter of logic}.

Notwithstanding the fact that the two variants of Kripke relational models (namely with and without impossible worlds) appeared almost at the same time, the state of development of their mathematical theory is not the same. In particular, although unsystematic correspondence results exist (viz.\ the ones in \cite{Kr65}), no Sahlqvist-type results are available for Kripke frames with impossible worlds. The present paper aims to extend state-of-the-art Sahlqvist theory to Kripke frames with impossible worlds.

\paragraph{State-of-the-art Sahlqvist theory.} Sahlqvist theory has a long history in normal modal logic, going back to \cite{Sa75} and \cite{van1983modal}. The Sahlqvist theorem in \cite{Sa75} gives a syntactic definition of a class of modal formulas, the {\em Sahlqvist class}, each member of which defines an elementary (i.e.\ first-order definable) class of frames and is canonical.

Over the years, many extensions, variations and analogues of this result have appeared. For instance, algebraic/topological proofs
\cite{SaVa89,Jonsson94}, constructive canonicity  \cite{GhMe97}, variations of the correspondence language  \cite{Benthem06}, Sahlqvist-type results for $\mu$-calculus  \cite{BenthemBH12}, significant syntactic enlargements of the Sahlqvist class  \cite{GorankoV06} and improved algorithmic methods in \cite{CoGoVa06}. In the last mentioned paper, the SQEMA algorithm is defined, which is guaranteed to succeed in computing the first-order correspondent of {\em inductive formulas}, these forming the widest syntactically defined class of modal formulas each member of which is guaranteed to enjoy the same properties as Sahlqvist formulas.

In \cite{ConPal12}, building on the insights discussed in \cite{ConPalSou12}, algorithmic canonicity and correspondence results analogous to those in \cite{CoGoVa06} have been obtained for the language of distributive modal logic, a modal logic framework the propositional base of which is the logic of distributive lattices. A critical methodological feature of this approach is the application of the dualities and adjunctions between the relational and the algebraic semantics of these logics. These have been used to distil the order-theoretic and algebraic significance of the SQEMA reduction steps, and hence to recast them into an algebraic setting which is more general than the Boolean. 

Taking stock of these results gave rise to the so-called {\em unified correspondence theory} \cite{CoGhPa14}, a framework within which correspondence results can be formulated and proved abstracting away from specific logical signatures, and only in terms of the order-theoretic properties of the algebraic interpretations of logical connectives. This has made it possible to uniformly export the state-of-the-art in Sahlqvist theory from normal modal logic to a wide range of logics. These logics include  substructural logics and any other logic algebraically captured by normal lattice expansions   \cite{ConPal13},   hybrid logics \cite{ConRob}, and bi-intuitionistic modal mu-calculus \cite{CoCr14,CoFoPaSo}.

The breadth of this work has also stimulated many and varied applications. Some of them are closely related to the core concerns of the theory itself, such as
the understanding of the relationship between different
methodologies for obtaining canonicity results \cite{PaSoZh14,CPZ:constructive}, or the phenomenon of pseudocorrespondence \cite{CGPSZ14}. Other, possibly surprising applications  include the dual characterizations of classes of finite lattices \cite{FrPaSa14}, the identification of the syntactic shape of axioms which can be translated into analytic structural rules of a proper display calculus \cite{GMPTZ}, and the definition of cut-free Gentzen calculi for subintuitionistic logics \cite{MaZhao15}. Finally, the insights of unified correspondence theory have made it possible to determine the extent to which the Sahlqvist theory of classes of normal DLEs can be reduced to the Sahlqvist theory of normal Boolean expansions, by means of G\"{o}del-type translations \cite{CPZ:Trans}.


\paragraph{Contributions and general organization.}
In the present paper, we apply the unified correspondence approach to obtain Sahlqvist-type canonicity and correspondence results about regular modal logic on classical, intuitionistic and distributive-lattice propositional bases. We mainly focus on two topics in Sahlqvist theory which have been developed independently of one another until very recently (cf.\ \cite{PaSoZh14}), namely {\em J\'onsson-style canonicity} (treated in Part I of the present paper) and {\em algorithmic correspondence and canonicity} (treated in Part II).

J\'onsson-style canonicity builds on the theory of canonical extensions, originating in \cite{JoTa51}. This method has been pioneered by J\'onsson in \cite{Jonsson94}, where the canonicity of Sahlqvist formulas of classical normal modal logic was proven in a purely algebraic way. Interestingly, this method does not rely on Sahlqvist correspondence, as other methodologies do (cf.\ e.g.\ \cite{SaVa89}). An expanded discussion on the relation between J\'onsson-style canonicity and other methodologies can be found in \cite{PaSoZh14}. J\'onsson's method for canonicity was also adopted in \cite{GeNaVe05}, in the setting of distributive modal logic (DML).

In Part I, we closely analyze the core of J\'onsson strategy as laid out in \cite{GeNaVe05}. As one consequence of this analysis, we find and emend some mistakes in the proof of canonicity given in \cite{GeNaVe05}. Specifically, we highlight that some steps in that proof rely on certain order-theoretic assumptions about the interpretations of the logical connectives which turn out to not be satisfied by all connectives involved; however, we show that J\'onsson's strategy goes through all the same under weaker assumptions, which are satisfied in the setting of \cite{GeNaVe05}. Namely, being an operator or a dual operator (that is, preserving or reversing joins or meets in each coordinate, including $\top$ and $\bot$) can be weakened to being {\em additive} or {\em multiplicative} (that is, preserving or reversing \emph{non-empty} joins or meets in each coordinate, while $\top$ and $\bot$ do not need to be preserved or reversed). Another consequence of this analysis is that the role played by smoothness in \cite{GeNaVe05} can be played by {\em stability} (see Section \ref{sec:5} for more discussion). The extension of J\'onsson-style canonicity from normal to regular modal logic is the third consequence of this analysis, given that additivity and multiplicativity are exactly the algebraic conditions guaranteeing the soundness and completeness of the logical axioms characterizing regular modal logics within classical modal logics. Specifically, using the toolbox of unified correspondence theory, we define Sahlqvist inequalities in the setting of regular modal logics (a class of inequalities covering Sahlqvist DML-inequalities defined in \cite{GeNaVe05}) and prove that they are canonical using J\'onsson's strategy.

As final conclusions emerging from our analysis, additivity and multiplicativity play a key role in the J\'onsson-style proof of canonicity of Sahlqvist inequalities both in the {\em normal} and in the {\em regular} setting. Interestingly, the normal case is not independent from the regular case.
So in a sense, the regular setting provides a kind of conceptual completion for J\'onsson-style canonicity.

Existing accounts of canonicity for additive connectives can also be found in \cite{Seki03}, in the context of expansions of relevant logics and Routley-Meyer type semantics, and in \cite{Suzuki11}. However, additivity versus normality is not their focal point, as it is in the present paper, and the correspondence part of the theory, when developed, refers to different semantics.

\medskip

In Part II, we introduce an adaptation, referred to as \textsf{ALBA}$^r$, of the calculus \textsf{ALBA} to regular modal logic (on weaker than classical bases). This adaptation builds on results in \cite{CGPSZ14}, namely, it is obtained by considering a certain restricted shape of some of the rules of the metacalculus \textsf{ALBA}$^e$ (cf.\ \cite[Section 5]{CGPSZ14}). We define the class of inductive inequalities in the regular setting. Again, this definition follows the principles of unified correspondence, and is given in terms of the order-theoretic properties of the algebraic interpretations of the logical connectives. Similar to the inductive inequalities defined in other settings, inductive DLR-inequalities properly and significantly extend Sahlqvist inequalities, while sharing their most important properties, namely the fact that the (regular) modal logics generated by them are strongly complete with respect to the class of Kripke frames defined by their first-order correspondent. We show that \textsf{ALBA}$^r$ succeeds on every inductive DLR-inequality. Part I and Part II can be read independently of each other.

We mentioned early on that \textsf{ALBA}$^r$ is very similar to \textsf{ALBA}$^e$. However, it is worth mentioning that they are different in important respects.
Firstly, the two settings of these algorithms (that is, the present setting and that of \cite{CGPSZ14}) are different: indeed,
the basic setting of \cite{CGPSZ14} is a {\em normal} modal logic setting (i.e.\ the primitive modal connectives are normal), but the term functions $\pi, \sigma, \lambda, \rho$ are assumed to be arbitrary compound formulas. Then, in \cite{CGPSZ14}, this basic setting is restricted even further to the class of DLEs on which the interpretations of $\pi, \sigma, \lambda, \rho$ verify additional (additivity or multiplicativity) conditions.

In contrast to this, in the present setting, the primitive connectives are not normal in the first place, but are assumed to be additive or multiplicative. Hence in particular the present basic setting covers a strictly wider class of algebras than the normal DLEs of \cite{CGPSZ14}.

Secondly, as mentioned above, \textsf{ALBA}$^r$ guarantees all the benefits of classical Sahlqvist correspondence theory for the inequalities on which it succeeds. This is not the case for \textsf{ALBA}$^e$ (cf.\ \cite[Section 9]{CGPSZ14}), which is used to prove relativized canonicity in the absence of correspondence. The reason for this difference is due to the fact that the approximation and the adjunction rules of \textsf{ALBA}$^e$ relative to compound term functions $\pi, \sigma, \lambda, \rho$ are only sound on perfect DLEs which are canonical extensions of some DLEs, whereas the corresponding rules for \textsf{ALBA}$^r$ concern only primitive regular connectives, and for this reason they can be shown to be sound on arbitrary perfect DLRs (cf.\ Theorem \ref{Crctns:Theorem}).

\medskip

In Part III, the previous parts come together, and are applied to our running examples. Namely, the strong completeness of Lemmon's logics E2-E5 with respect to elementary classes of Kripke frames with impossible worlds is obtained as a consequence of the theory developed in Parts I and II, and the defining first-order conditions are effectively computed via \textsf{ALBA}$^r$. This result is also obtained as an application of the theory.

\section{Impossible worlds and non-normal modal logics}\label{sec:prelimi}
In the present section, we report on historically important examples of non-normal modal logics, namely Lemmon's E2-E4, which have been given semantic interpretation in terms of Kripke frames with impossible worlds. We also collect preliminaries on Kripke frames with impossible worlds and their complex algebras. We briefly mention how the discrete duality on objects between usual Kripke frames and Boolean algebras with operators can be extended to Kripke frames with impossible worlds and perfect regular Boolean algebra expansions (r-BAEs, cf.\ Definition \ref{def:regular BAE}). Finally, we outline the generalization of this discrete duality to the distributive lattice based counterparts of r-BAEs and the poset-based counterparts of Kripke frames with impossible worlds.

\subsection{Regular modal logics}\label{subset:reg:modal:logic}
\paragraph{} \emph{Classical modal logics} (cf.\ \cite{CH90}, Definition 8.1, and \cite{SG71}) are weaker than normal modal logics, and are only required to contain the axiom $\Diamond A\leftrightarrow \neg \Box\neg A$ and be closed under the following rule:
\vspace{2mm}

\AxiomC{$ A\leftrightarrow B$}
(RE)\UnaryInfC{$\Box A\leftrightarrow \Box B $}
\DisplayProof
\vspace{2mm}

\emph{Monotonic modal logics} are required to contain the axiom above and be closed under the following rule:
\vspace{2mm}

\AxiomC{$ A\rightarrow B$}
(RM)\UnaryInfC{$\Box A\rightarrow \Box B $}
\DisplayProof

\vspace{2mm}

\emph{Regular modal logics} (cf.\ \cite{CH90}, Definition 8.8)
 are required to contain the axiom above and be closed under the following rule:
\vspace{2mm}

\AxiomC{$ (A\wedge B)\rightarrow C$}
(RR)\UnaryInfC{$ (\Box A\wedge \Box B)\rightarrow \Box C $}
\DisplayProof

\vspace{2mm}
It is important to note that (RE) can be derived from either (RR) or (RM), and hence both monotonic modal logics and regular modal logics are classical (cf.\ Corollary \ref{cor:regular}). Notice also that (RR) can be derived from (RM) in the presence of the axiom $(\Box p\wedge \Box q) \rightarrow \Box (p\wedge q)$:

\begin{lemma}
\label{fact:RR:RM}
In the presence of \emph{(RM)} and $(\Box p\wedge \Box q) \rightarrow \Box (p\wedge q)$, we have \emph{(RR)}.
\end{lemma}
\begin{proof}
Assume $\vdash(A\wedge B)\rightarrow C$. By (RM), it follows that $\vdash\Box(A\wedge B)\rightarrow\Box C$. By axiom $(\Box p\wedge \Box q) \rightarrow \Box (p\wedge q)$, we have $(\Box A\wedge \Box B) \rightarrow \Box (A\wedge B)$. Therefore we get $\vdash(\Box A\wedge \Box B)\rightarrow \Box C$.
\end{proof}

Hence, regular modal logics can be equivalently defined as monotonic modal logics which contain the axiom $(\Box p\wedge \Box q) \rightarrow \Box (p\wedge q)$.

In what follows, we will consider as running examples some historically important modal logics which have been given a semantic interpretation in terms of Kripke frames with impossible worlds. These are Lemmon's logics E2-E5.

\paragraph{Lemmon's epistemic logics.} In \cite{Le57}, the systems E1-E5 have been introduced as the epistemic counterparts of Lewis' modal systems S1-S5. The E-systems E2-E4 are examples of regular but not normal modal logics, while E5 turns out to coincide with the normal modal logic S5 (as mentioned in \cite[p.\ 209]{Kr65}). A semantic proof of this fact is given in Section \ref{sec:example:lemmon:corre:canoni}. With the E-systems, Lemmon intended to capture the principle that ``nothing is a scientific law as a matter of logic''. In particular, Lemmon argues that the necessitation rule is not plausible if the modal operator $\Box$ is to be interpreted as ``scientific but not logical necessity''.
The following axioms and rules are reported here with the same names as in \cite{Le57}. The notation has been changed to the currently standard one.\\

Axioms:

\begin{tabular}{llll}
\vspace{2mm}
(P) & Propositional tautologies &&\\
\vspace{2mm}
(1) & $\Box(p\to q)\to\Box(\Box p\to\Box q)$ & (1') & $\Box(p\to q)\to(\Box p\to\Box q)$\\
\vspace{2mm}
(2) & $\Box p\to p$ & & \\
\vspace{2mm}
(4) & $\Box p\to\Box\Box p$ & (5) & $\neg\Box p\to\Box\neg\Box p$.\\
\end{tabular}

Rules:

\begin{enumerate}
\item[(PCa)] If $\alpha$ is a propositional tautology, then $\vdash\alpha$.
\item[(PCb)] Substitution for proposition variables.
\item[(PCc)] Modus Ponens.
\item[(Eb)] If $\vdash\alpha\to\beta$, then $\vdash\Box\alpha\to\Box\beta$.
\end{enumerate}
\label{pageref:lemmonsystem}The logics E2-E5 are defined as follows:\\

\begin{tabular}{l c l}

\vspace{2mm}

PC:\; (P)+(PCa)+(PCb)+(PCc) &&\\

\vspace{2mm}

E2:\; PC+(Eb)+(1')+(2) &$\quad$ & E3:\; PC+(Eb)+(1)+(2) \\

\vspace{2mm}

E4:\; E2+(4) && E5:\; E2+(5) \\
\end{tabular}

\begin{lemma}
\label{fact:regular}
 In the presence of \emph{(Eb)} and \emph{(1')}, we have $\vdash\Box(\alpha\land\beta)\leftrightarrow\Box\alpha\land\Box\beta$.
 \end{lemma}
\begin{proof}
 The left-to-right implication is obtained by applying (Eb) to propositional tautologies $\alpha\land\beta\to\alpha$ and $\alpha\land\beta\to\beta$, which yields $\vdash\Box(\alpha\land\beta)\to\Box\alpha$ and $\vdash\Box(\alpha\land\beta)\to\Box\beta$. From these, PC derives $\vdash\Box(\alpha\land\beta)\to\Box\alpha\land\Box\beta$.

As to the right-to-left implication, it is enough to show that $\Box \alpha \vdash \Box\beta\to\Box(\alpha\land\beta)$. Indeed:
\[ \Box \alpha \vdash^{(1)} \Box(\beta\to(\alpha\land\beta))\vdash^{(2)} \Box\beta\to\Box(\alpha\land\beta). \]
Entailment $\vdash^{(1)}$ follows from $\vdash\Box\alpha\to\Box(\beta\to(\alpha\land\beta))$, which can be derived by applying (Eb) to the propositional tautology $\vdash \alpha\to (\beta\to (\alpha\wedge \beta))$.

Entailment $\vdash^{(2)}$ follows from $\vdash\Box(\beta\to(\alpha\land\beta))\to(\Box\beta\to\Box(\alpha\land\beta))$, which is obtained as a suitable instantiation of (1').
\end{proof}
\begin{cor}\label{cor:regular}
Lemmon's logics E2-E5 are regular.
\end{cor}
\begin{proof}
Notice preliminarily that $(1)$ and $(2)$ entail $(1')$. Indeed,
 \[\Box(\alpha\to \beta)\vdash^{(1)}\Box(\Box\alpha\to \Box\beta)\vdash^{(2)} \Box\alpha\to \Box\beta.\]
 Entailment $\vdash^{(1)}$ immediately follows from assumption (1) above. 
 Entailment $\vdash^{(2)}$ immediately follows from $\vdash\Box(\Box\alpha\to \Box\beta)\to (\Box\alpha\to \Box\beta)$, which is a suitable instantiation of (2). Hence, the regularity of E2-E5 immediately follows from Lemma \ref{fact:regular}.
\end{proof}

\subsection{Kripke frames with impossible worlds and their complex algebras}
\begin{definition}[cf.\ \cite{Kr65, Be13}]
A {\em Kripke frame with impossible worlds} is a triple $\mathbb{F} = (W, S, N)$ such that $W\neq \varnothing$, $N\subseteq W$ and $S\subseteq N\times W$. The set $N$ is regarded as the collection of so-called {\em normal worlds}, i.e.\ $W\setminus N$ is the collection of impossible worlds. A {\em Kripke model with impossible worlds} is a pair $\mathbb{M}=(\mathbb{F}, V)$ such that $\mathbb{F}$ is a Kripke frame with impossible worlds, and $V:\mathsf{Prop}\to\mathcal{P}(W)$ is an assignment.
\end{definition}
The satisfaction of atomic formulas and of formulas the main connective of which is a Boolean connective is defined as usual. As to $\Box$- and $\Diamond$-formulas,
\[\mathbb{M}, w\Vdash \Box\phi\quad \mbox{ iff }\quad w\in N\ \mbox{ and for all }\ v\in W \mbox{ if } wS v \mbox{ then }\mathbb{M}, v\Vdash\phi.\]
\[\mathbb{M}, w\Vdash \Diamond\phi\quad\mbox{ iff }\quad w\notin N\mbox{ or there exists some }v\in W \mbox{ s.t. }wS v, \mbox{ and } \mathbb{M}, v\Vdash\phi.\]
 Validity of formulas is defined as usual.

The definition above is tailored to make the necessitation rule fail. Indeed, $\top$ is valid at every point in every model, however $\mathcal{M}, w\not\Vdash \Box\top$ whenever $w\notin N$.
\begin{lemma}
\label{fact:multiplicativity}
Axiom $(\Box p\land \Box q)\leftrightarrow\Box(p\land q)$ is valid on any Kripke frame with impossible worlds.
\end{lemma}

\begin{proof}

Let $\mathbb{F}=(W, S, N)$ be a Kripke frame with impossible worlds. Fix a valuation $V$ and let $w\in W$. If $w\notin N$, then $w\nVdash\Box p$, $w\nVdash\Box q$, $w\nVdash\Box(p\land q)$. So $w\Vdash(\Box p\land \Box q)\leftrightarrow\Box(p\land q)$. If $w\in N$, the following chain of equivalences holds:

\vspace{2mm}
\begin{center}
\begin{tabular}{ll}
&$w\Vdash\Box p\land \Box q$ \\
iff &$w\Vdash\Box p$ and $w\Vdash\Box q$ \\
iff & for all $v\in W$ s.t.\ $wSv$, $v\Vdash p$ and for all $v\in W$ s.t.\ $wSv$, $v\Vdash q$ \\
iff & for all $v\in W$ s.t.\ $wSv$, $v\Vdash p$ and $v\Vdash q$\\
iff &$w\Vdash\Box (p\land q)$.\\
\end{tabular}
\end{center}
\end{proof}
\begin{definition}
\label{def: complex alg kfimpossible}
For any Kripke frame with impossible worlds $\mathbb{F} = (W, S, N)$, the {\em complex algebra} associated with $\mathbb{F}$ is $\mathbb{F}^+: = (\mathcal{P}(W), \Diamond_S)$ where $\Diamond_S: \mathcal{P}(W)\to \mathcal{P}(W)$ is defined by the following assignment: \[X\mapsto \{w\in W\mid w\notin N \mbox{ or } wSv \mbox{ for some } v\in X\} = N^c\cup S^{-1}[X]. \]
\end{definition}
\begin{lemma}
\label{fact:f_S compl aditive}
For any Kripke frame with impossible worlds $\mathbb{F} = (W, S, N)$ and any $\mathcal{X}\subseteq \mathcal{P}(W)$, if $\mathcal{X}\neq \varnothing$, then \[\Diamond_S(\bigcup\mathcal{X}) = \bigcup\{\Diamond_S(X)\mid X\in \mathcal{X}\}.\]
\end{lemma}
The fact above shows that the diamond-type operation of the complex algebra associated with any Kripke frame with impossible worlds is completely additive.
Lemmas \ref{fact:multiplicativity} and \ref{fact:f_S compl aditive} witness that the class of Kripke frames with impossible worlds form a natural semantic environment for the so-called regular modal logics, the definition of which is introduced below.

\subsection{Algebraic semantics}
In the present subsection, we collect basic facts about algebraic semantics for (Boolean-based) regular modal logics.

\begin{definition}
\label{def:regular BAE}
A {\em regular Boolean algebra expansion} (from now on abbreviated as r-BAE) is a tuple $\mathbb{A} = (\mathbb{B}, f)$ such that $\mathbb{B}$ is a Boolean algebra, and $f$ is a unary and additive operation, i.e.\ $f$ preserves finite nonempty joins of $\mathbb{B}$. An r-BAE is {\em perfect} if $\mathbb{B}$ is complete and atomic, and $f$ is completely additive, i.e.\ $f$ preserves arbitrary nonempty joins of $\mathbb{B}$.
\end{definition}
The complex algebra associated with every Kripke frame with impossible worlds is a perfect r-BAE (cf.\ Lemma \ref{fact:f_S compl aditive}).

Formulas in the language of regular modal logic are interpreted into r-BAEs via assignments in the usual way. An \emph{assignment} is a function $h:\mathsf{Prop}\to\A$. Each assignment has a unique homomorphic extension to the algebra of formulas over $\mathsf{Prop}$, which we also denote by $h$.
An \emph{equation} is a pair of formulas $(s,t)$, usually written as $s \approx t$. An equality $s \approx t$ is valid in an r-BAE $\A$ (notation: $\A \models s \approx t$) if $h(s) =h(t)$ for all assignments $h$. An \emph{inequality} is a pair of formulas $(s,t)$, usually written as $s \leq t$.  An inequality $s \leq t$ is valid on an r-BAE $\A$ (notation: $\A \models s \leq t$) if $h(s) \leq^{\A} h(t)$ for all assignments $h$, where $\leq^{\A}$ is the lattice order of $\A$. If $\K$ is a class of r-BAEs, then $s \approx t$ is valid on $\K$ (notation: $\K \models s \approx t$) if $s \approx t$ is valid in every $\A\in \K$. The validity of $s\leq t$ in $\K$ is defined similarly.

It is well known (cf.\ \cite[Theorem 7.5]{HH03}) that the basic regular modal logic is sound and complete with respect to the class of r-BAEs.

The well known duality on objects between Kripke frames and complete atomic Boolean algebras with complete operators generalizes to the setting of perfect r-BAEs and Kripke frames with impossible worlds. Indeed, Definition \ref{def: complex alg kfimpossible} provides half of the connection. As to the other half:

\begin{definition}
For every perfect r-BAE $\mathbb{A} = (\mathbb{B}, f)$, the {\em atom structure with impossible worlds} associated with $\mathbb{A}$ is $\mathbb{A}_{+}: = (At(\mathbb{A}), S_f, N)$, where $At(\mathbb{A})$ is the collection of atoms of $\mathbb{A}$, $N: = \{x\in At(\mathbb{A})\mid x\nleq f(\bot) \}$ and for all $x, y\in At(\mathbb{A})$ such that $x\nleq f(\bot)$,
\[xS_f y \quad \mbox{ iff }\quad x\leq f(y). \]
\end{definition}
We could not find a precise reference to the following proposition in the literature, so it should perhaps be considered  folklore.
\begin{prop}
For every Kripke frame with impossible worlds $\mathbb{F}$ and every perfect r-BAE $\mathbb{A}$, \[
(\mathbb{F}^{+})_{+}\cong \mathbb{F}\quad \mbox{ and } \quad (\mathbb{A}_{+})^{+}\cong \mathbb{A}.\]
\end{prop}

\subsection{The distributive setting}
In the present subsection, we briefly outline the BDL-versions of the definitions and facts of the previous subsection. The reason for this generalization is that most of the treatment in Parts I and II is presented in the setting of distributive lattices.

\begin{definition}
A {\em regular distributive lattice expansion}\footnote{Actually, this name refers much more generally to bounded distributive lattices with an arbitrary but finite number of additional operations of any finite arity. Hence, the definition given above captures only a very restricted subclass of distributive lattice expansions which is sufficient for the sake of the present section. Later on in the paper, we will introduce several other proper subclasses of distributive lattice expansions. However, for the sake of brevity, we will abuse terminology and also refer to them as ``distributive lattice expansions'' without any further specification (cf.\ Definition \ref{def:dlr}).} (abbreviated as r-DLE in the remainder of the section) is a tuple $\mathbb{A} = (\mathbb{B}, f, g)$\footnote{In fact, here we can also add unary maps which turn finite non-empty joins (resp.\ meets) into meets (resp.\ joins). Here we only talk about unary and additive (resp.\ multiplicative) operations without loss of generality.} such that $\mathbb{B}$ is a bounded distributive lattice, and $f$ (resp.\ $g$) is a unary and additive (resp.\ multiplicative) operation,  i.e.\ it preserves finite nonempty joins (resp.\ meets) of $\mathbb{B}$. An r-DLE is {\em perfect} if $\mathbb{B}$ is a perfect distributive lattice, i.e.\ it is complete, completely distributive and completely join-generated by the collection of its completely join-prime elements, and $f$ (resp.\ $g$) is completely additive (resp.\ completely multiplicative), i.e.\ it preserves arbitrary nonempty joins (resp.\ arbitrary nonempty meets) of $\mathbb{B}$. An $x\in \A$ is {\em completely join-irreducible} (resp.\ {\em completely join-prime}) if $x\neq \bot$ and for every $A\subseteq \A$, if $x = \bigvee S$ (resp.\ $x \leq \bigvee S$) then $x = s$ (resp.\ $x\leq s$) for some $s\in S$. In the setting of perfect distributive lattices, completely join-irreducibles coincide with completely join-primes.
\end{definition}

\begin{definition}
A {\em distributive Kripke frame with impossible worlds} is a quadruple $\mathbb{F} = (W, \leq, S_f, S_g, N_f, N_g)$ such that $W$ is a nonempty set, $\leq$ is a partial order on $W$, the subsets $N_f, N_g\subseteq W$ are a down-set and an up-set respectively (the collections of {\em $f$-normal worlds} and {\em $g$-normal worlds} respectively), $S_f\subseteq N_f\times W, S_g\subseteq N_g\times W$ are such that
$${\geq_f}\circ S_f\circ {\geq}\subseteq S_f\quad \mbox{ and }\quad {\leq_g}\circ S_g\circ {\leq}\subseteq S_g,$$
where $\geq_f$ and $\leq_g$ respectively denote the restrictions of $\geq$ and $\leq$ to $N_f$ and $N_g$.
\end{definition}

For every perfect r-DLE $\mathbb{A} = (\mathbb{B}, f, g)$, let $\mathbb{A}_{+} : = (\jty(\mathbb{A}), \geq, S_f, S_g, N_f, N_g)$ where $\jty(\mathbb{A})$ is the collection of completely join-irreducible elements of $\mathbb{A}$, the symbol $\geq$ denotes the converse of the lattice order restricted to $\jty(\mathbb{A})$, $N_f: = \{x\in \jty(\mathbb{A})\mid x\nleq f(\bot) \}$, $N_g: = \{x\in \jty(\mathbb{A})\mid x\leq g(\top) \}$, and $S_f\subseteq N_f\times \jty(\mathbb{A})$ and $S_g\subseteq N_g\times \jty(\mathbb{A})$ are binary relations defined as follows: 
\[xS_f y \quad \mbox{ iff }\quad x\leq f(y)\qquad \mbox{ and }\qquad\ xS_g y \quad \mbox{ iff }\quad \kappa(x)\geq g(\kappa(y)),\]

where $\kappa:\jty(\mathbb{A})\to \mty(\mathbb{A})$ is the order isomorphism defined by $\kappa(x)=\bigvee\{a\in \mathbb{A}\mid x\nleq a\}$.

 For every distributive Kripke frame with impossible worlds $\mathbb{F} = (W, \leq , S_f, S_g, N_f, N_g)$, let
$\mathbb{F}^+: = (\mathcal{P}^{\uparrow}(W), f, g)$ where $\mathcal{P}^{\uparrow}(W)$ is the set of the upward-closed subsets of $W$, and $f, g: \mathcal{P}^{\uparrow}(W)\to \mathcal{P}^{\uparrow}(W)$ are respectively defined by the following assignments:

\[f(X)=\{w\in W\mid w\notin N_f \mbox{ or } v\in X \mbox{ for some } v \mbox{ s.t.\ }wS_f v\} = N_f^c\cup S_f^{-1}[X]. \]
\[g(X)=\{w\in W\mid w\in N_g \mbox{ and } v\in X \mbox{ for all } v \mbox{ s.t.\ }wS_g v\} = N_g\cap (S_g^{-1}[X^c])^c. \]

We could not find a precise reference to the following proposition in the literature, so it should perhaps be considered  folklore.

\begin{prop}
For every distributive Kripke frame with impossible worlds $\mathbb{F}$ and every perfect r-DLE $\mathbb{A}$, \[
(\mathbb{F}^{+})_{+}\cong \mathbb{F}\quad \mbox{ and } \quad (\mathbb{A}_{+})^{+}\cong \mathbb{A}.\]
\end{prop}

\section*{Part I: J\'onsson-style canonicity}
\addcontentsline{toc}{section}{\underline{Part I: J\'onsson-style canonicity}}

\section{Preliminaries on algebraic canonicity}

Technically, the starting point of the results of the present paper is the analysis of the proof in \cite{GeNaVe05}, reported in the present section. Subsections \ref{Subsec:can:ext:BDL} and \ref{subsec:can:ext:maps} collects the basic definitions and preliminaries on canonical extensions of bounded distributive lattices, and on $\sigma$- and $\pi$-extensions of maps. In Subsection \ref{ssec:continuity properties}, we report  and discuss some properties of the $\sigma$- and $\pi$-extensions which are pivotal to the J\'onsson-style argument for canonicity. In Subsection \ref{subsec:dist:latt:expansions}, we collect preliminaries on a slightly modified formulation of the setting of Distributive Modal Logic, which gives rise to a class of languages, indexed by certain parameters. In Subsection \ref{subsec:DLE:eta}, applying the methodology of unified correspondence, we define the Sahlqvist classes relative to each of the languages defined in Subsection \ref{subsec:dist:latt:expansions}.

\subsection{Canonical extensions of bounded distributive lattices}
\label{Subsec:can:ext:BDL}
In the present section, $\A, \B, \C$, possibly with sub- or superscripts, denote bounded distributive lattices (BDLs).

For any BDL $\mathbb{A}=(A, \lor, \land, \perp, \top)$, let $\mathbb{A}^{\partial} := (A, \land, \lor, \top, \perp)$ be its \emph{dual} BDL. Moreover, let $\mathbb{A}^1: =\mathbb{A}$.

An \emph{$n$-order type} $\epsilon$ is an element of $\{1, \partial\}^n$, and its $i$-th coordinate is denoted $\epsilon_i$. We omit $n$ when it is clear from the context. Let $\epsilon^\partial$ denote the \emph{dual order type} of $\epsilon$, that is, $\epsilon^\partial_i: =1$ (resp.\ $\epsilon^\partial_i: =\partial$) if $\epsilon_i=\partial$ (resp.\ $\epsilon_i=1$). Given an order type $\epsilon$, we let $\mathbb{A}^{\epsilon}$ be the BDL $\mathbb{A}^{\epsilon_1}\times\ldots\times\mathbb{A}^{\epsilon_n}$. For an order type $\epsilon$ and $\overline{a}, \overline{b}\in \A^\epsilon$, we let $\overline{a}\leq_{\epsilon}\overline{b} $ iff for each $1\leq i\leq n$, we have $a_i\leq b_i$ if $\epsilon _i = 1$ and $b_i\leq a_i$ if $\epsilon _i = \partial$.

\begin{definition}\label{def:order theoretic properties}
Let $1\leq i\leq n$. An $n$-ary map $f:\Pi_i\mathbb{A}_i\rightarrow\mathbb{B}$ is:\footnote{This terminology is taken from \cite{HBMoL}.}
\begin{itemize}
\item \emph{additive} if $f$ preserves non-empty finite joins in each coordinate.
\item \emph{multiplicative} if $f$ preserves non-empty finite meets in each coordinate.
\item \emph{p-additive} if $f$ preserves non-empty finite joins in $\Pi_i\mathbb{A}_i$.
\item \emph{p-multiplicative} if $f$ preserves non-empty finite meets in $\Pi_i\mathbb{A}_i$.
\item \emph{normal} when for any $\overline{a}\in \Pi_i\mathbb{A}_i$, if $a_i=\bot$ for some $1\leq i\leq n$, then $f(\overline{a})=\bot$.
\item \emph{dually normal} when for any $\overline{a}\in\Pi_i\mathbb{A}_i$, if $a_i=\top$ for some $1\leq i\leq n$, then $f(\overline{a})=\top$.
\item an \emph{operator} if $f$ is additive and normal.
\item a \emph{dual operator} if $f$ is multiplicative and dually normal.
\item \emph{join-preserving} if $f$ preserves arbitrary finite joins\footnote{Hence also $\bot=\bigvee\varnothing$ is preserved.} in the product.
\item \emph{meet-preserving} if $f$ preserves arbitrary finite meets\footnote{Hence also $\top=\bigwedge\varnothing$ is preserved.} in the product.
\item \emph{completely additive} if $f$ preserves non-empty arbitrary joins in each coordinate.
\item \emph{completely multiplicative} if $f$ preserves non-empty arbitrary meets in each coordinate.
\item \emph{completely p-additive} if $f$ preserves non-empty arbitrary joins in $\Pi_i\mathbb{A}_i$.
\item \emph{completely p-multiplicative} if $f$ preserves non-empty arbitrary meets in $\Pi_i\mathbb{A}_i$.
\item a \emph{complete operator} if $f$ preserves arbitrary joins in each coordinate.
\item a \emph{complete dual operator} if $f$ preserves arbitrary meets in each coordinate.
\item \emph{completely join-preserving} if $f$ preserves arbitrary joins in the product.
\item \emph{completely meet-preserving} if $f$ preserves arbitrary meets in the product.
\end{itemize}
\end{definition}

\begin{definition}\label{def:can:ext2.3}
The \emph{canonical extension} of a BDL $\mathbb{A}$ is a complete BDL $\mathbb{A}^\delta$ containing $\mathbb{A}$ as a sublattice, such that:
\begin{enumerate}
\item \emph{(denseness)} every element of $\mathbb{A}^\delta$ can be expressed both as a join of meets and as a meet of joins of elements from $\mathbb{A}$;
\item \emph{(compactness)} for all $S,T \subseteq \mathbb{A}$ with $\bigwedge S \leq \bigvee T$ in $\mathbb{A}^\delta$, there exist some finite sets $F \subseteq S$ and $G\subseteq T$ s.t.\ $\bigwedge F \leq \bigvee G$.
\end{enumerate}
\end{definition}

It is well known that the canonical extension of a BDL is unique up to isomorphism (cf.\ \cite[Section 2.2]{GeNaVe05}), and that the canonical extension of a BDL is a perfect BDL (cf.\ \cite[Definition 2.14]{GeNaVe05})\label{canext bdl is perfect}:

\begin{definition}
A BDL $\bba$ is \emph{perfect} if it is complete, completely distributive, and completely join-generated by the set $J^{\infty}(\bba)$ of the completely join-irreducible elements of $\mathbb{A}$, as well as completely meet generated by the set $M^{\infty}(\bba)$ of the completely meet-irreducible elements of $\mathbb{A}$.
\end{definition}

An element $x \in \ca$ is \emph{closed} (resp.\ \emph{open}) if it is the meet (resp.\ join) of some subset of $\A$. The set of closed (resp.\ open) elements of $\ca$ is denoted by $K(\ca)$ (resp.\ $O(\ca)$). It is easy to see that the denseness condition in Definition \ref{def:can:ext2.3} implies that $J^{\infty}(\bbas)\subseteq K (\mathbb{A}^\delta)$ and $M^{\infty}(\bbas)\subseteq O (\mathbb{A}^\delta)$ (cf.\ \cite{GeNaVe05}, page 9). Recall that any $j\in J^{\infty}({\A^{\delta}}^n)$ is of the form $\overline{x}$ with $x = \bot$ in all coordinates but one, in which $x\in J^{\infty}(\bba)$  (cf.\ \cite{DaPr90}, Section 5.15).

The following properties can be easily checked about the interaction among canonical extension, order duals, and products:
\begin{lemma}[cf.\ \cite{GeNaVe05}, Section 2.2] For all BDLs $\A, \A_i$ with $1\leq i\leq n$,
\begin{enumerate}
\item $(\mathbb{A}^{\partial})^{\delta}\cong(\mathbb{A}^{\delta})^{\partial}$;
\item $(\Pi_{i = 1}^n\mathbb{A}_i)^{\delta}\cong\Pi_{i = 1}^n(\mathbb{A}_i^{\delta})$;
\item $K((\mathbb{A}^{\partial})^{\delta})=O(\mathbb{A}^{\delta})^{\partial}$;
\item $O((\mathbb{A}^{\partial})^{\delta})=K(\mathbb{A}^{\delta})^{\partial}$;
\item $K((\Pi_{i = 1}^n\mathbb{A}_i)^{\delta})=\Pi_{i = 1}^n K(\mathbb{A}_i^{\delta})$;
\item $O((\Pi_{i = 1}^n\mathbb{A}_i)^{\delta})=\Pi_{i = 1}^n O(\mathbb{A}_i^{\delta})$.
\end{enumerate}

\end{lemma}

\subsection{Canonical extensions of maps}\label{subsec:can:ext:maps}

Let $\mathbb{A},\mathbb{B}$ be BDLs. A monotone map $f : \mathbb{A}\rightarrow \mathbb{B}$ can be extended to a map $\ca\to\B^{\delta}$ in two canonical ways. Let $f^\sigma$ and $f^\pi$ respectively denote the $\sigma$- and the $\pi$\emph{-extension} of $f$, defined as follows:

\begin{definition}\label{def:canonical:extension:maps}
For any monotone map $f : \mathbb{A}\rightarrow\mathbb{B}$ and any $u\in \mathbb{A}^\delta$, let
\[f^\sigma (u) : =\bigvee \{ \bigwedge \{f(a): x\leq a\in \mathbb{A}\}: u \geq x \in K (\ca)\}\]
\[f^\pi (u): =\bigwedge \{ \bigvee \{f(a): y\geq a\in \mathbb{A}\}: u \leq y \in O (\ca)\}.\]
\end{definition}
The definitions above readily imply that $f^\sigma\leq f^\pi$\label{fsigmaleqfpi}.
\begin{definition}

A monotone map $f : \mathbb{A}\rightarrow\mathbb{B}$ is \emph{smooth} if $f^\sigma(u)=f^\pi(u)$ for every $u\in\A$.

\end{definition}

For any $f:\mathbb{A}\to\mathbb{B}$, let $f^{\partial}:\mathbb{A}^{\partial}\to\mathbb{B}^{\partial}$ be defined by $f^{\partial}(a):=f(a)$ for all $a\in\A$. Hence, it is not difficult to see that $(f^{\partial})^{\sigma}=(f^{\pi})^{\partial}$ and $(f^{\partial})^{\pi}=(f^{\sigma})^{\partial}$.

The following lemma slightly refines \cite[Lemma 2.22]{GeJo04} in the sense that additive (resp.\ multiplicative) maps do not necessarily preserve empty join (resp.\ meet). In what follows, we will find it useful to let the symbol $J^{\infty}_{\bot}(\mathbb{A}^{\delta})$ abbreviate the set $J^{\infty}(\mathbb{A}^{\delta})\cup \{\bot\}$ for any BDL $\mathbb{A}$.

\begin{lemma}\label{lem:complete:additive}

For any monotone map $f:\Pi\A_m\to\B$,

\begin{enumerate}
\item if $f$ is additive, then $f^{\sigma}:\Pi\A^{\delta}_m\to\B^{\delta}$ is completely additive;
\item if $f$ is multiplicative, then $f^{\pi}:\Pi\A^{\delta}_m\to\B^{\delta}$ is completely multiplicative.
\item if $f$ is an operator, then $f^{\sigma}:\Pi\A^{\delta}_m\to\B^{\delta}$ is a complete operator;
\item if $f$ is a dual operator, then $f^{\pi}:\Pi\A^{\delta}_m\to\B^{\delta}$ is a complete dual operator.
\end{enumerate}

\end{lemma}

\begin{proof}
(i). Fix any coordinate $1\leq i\leq n$. Let $\overline{u}[x/u^{(i)}]$ denote the substitution of $u$ by $x$ in the $i^{th}$ coordinate of $\overline{u}$. It is enough to show that, for any $\overline{u}\in \Pi\mathbb{A}^{\delta}_m$ and any $q\in J^{\infty}(\mathbb{B}^{\delta})$, if $q\leq f^{\sigma}(\overline{u})$, then $q\leq f^{\sigma}(\overline{u}[x/u^{(i)}])$ for some $x\in J^{\infty}_{\bot}(\mathbb{A}_i^\delta)$ such that $x\leq u^{(i)}$.
By definition, $f^{\sigma}(\overline{u}) = \bigvee\{f^{\sigma}(\overline{k})\mid \overline{k}\in \Pi K(\mathbb{A}^{\delta}_m) \mbox{ and } \overline{k}\leq \overline{u}\}$, so we can assume without loss of generality that $\overline{u}\in \Pi K(\mathbb{A}^{\delta}_m)$. Assume contrapositively that $q\nleq f^{\sigma}(\overline{u}[x/u^{(i)}])$ for any $x\in J^{\infty}_{\bot}(\mathbb{A}_i^\delta)$ such that $x\leq u^{(i)}$. So for any such $x$ there exists some $\overline{a}_x\in \Pi\mathbb{A}_m$ s.t.\ $\overline{u}[x/u^{(i)}]\leq \overline{a}_x$ and $q\nleq f^{\sigma}(\overline{a}_x)$. Since $$K(\mathbb{A}^{\delta}_i)\ni u^{(i)} =\bigvee\{x\in J^{\infty}_{\bot}(\mathbb{A}_i^\delta) \mid x\leq u^{(i)}\}\leq \bigvee\{a^{(i)}_x\mid x\in J^{\infty}_{\bot}(\mathbb{A}_i^\delta) \mbox{ and } x\leq u^{(i)}\},$$
by compactness, $ u^{(i)} \leq \bigvee\{a^{(i)}_x\mid x\in F \}$
for some finite and non-empty $F \subseteq \{x\in J^{\infty}_{\bot}(\mathbb{A}_i^\delta) \mid x\leq u^{(i)}\}$. Let $a = \bigvee\{a^{(i)}_x\mid x\in F\}$ and let $b^{(j)} = \bigwedge\{a^{(j)}_x\mid x\in F\}$ for any $j\neq i$. Clearly $\overline{b}[a/\bot_i]\in \Pi \mathbb{A}_m$, and by definition, $f^{\sigma}(\overline{u})\leq f(\overline{b}[a/\bot_i])$. Moreover, by the monotonicity of $f^{\sigma}$, it follows that $q\nleq f^\sigma(\overline{b}[a^{(i)}_x/\bot_i]) = f(\overline{b}[a^{(i)}_x/\bot_i])$ for each $x\in F$: indeed, if $q\leq f^\sigma(\overline{b}[a^{(i)}_x/\bot_i])$, then from $\overline{b}[a^{(i)}_x/\bot_i]\leq\overline{a}_x$ we get that $q\leq f^{\sigma}(\overline{a}_x)$, a contradiction. Since $f$ is additive in its $i$-th coordinate, $q\nleq \bigvee\{ f(\overline{b}[a^{(i)}_x/\bot_i])\mid x\in F\} = f(\overline{b}[\bigvee\{a^{(i)}_x\mid x\in F\}/\bot_i]) = f(\overline{b}[a/\bot_i])$. Since $f^{\sigma}(\overline{u})\leq f(\overline{b}[a/\bot_i])$, this implies that $q\nleq f^{\sigma}(\overline{u})$, as required.

Item (ii) is order dual to (i). Items (iii) and (iv) immediately follow from (i) and (ii) respectively, and the identities $f^{\sigma}(\bot) = f(\bot)$ and $f^{\sigma}(\top) = f(\top)$.
\end{proof}

\begin{theorem}[cf.\ \cite{Ri52}]\label{thm:ribiero}
For all DL maps $f:\A\rightarrow\B$ and $g:\B\rightarrow\C$, if $f$ and $g$ are monotone, then $(g f)^\sigma\leq g^\sigma f^\sigma$.

\end{theorem}

\subsection{Continuity properties}
\label{ssec:continuity properties}
In the present subsection, we report on and discuss some properties of $\sigma$- and $\pi$-extensions which are pivotal to the J\'onsson-style argument for canonicity. In the discussion, we correct some small mistakes in \cite{GeNaVe05}, the emendation of which is related to the ensuing generalization in Sections \ref{sec:weaker sufficient conditions} and \ref{sec:5}.

In what follows, we let $J^{\infty}_{\omega}(\A^{\delta}):=\bigcup \{J^{\infty}_n(\A^{\delta})\mid n\in\mathbb{N}\}$, where, for every $n\in\mathbb{N}$, the set $J^{\infty}_{n}(\A^{\delta})$ contains all $\bigvee J$ such that $J\subseteq J^{\infty}(\A^{\delta})$ and $|J|\leq n$.

\begin{definition}
\label{def:UC}
Any monotone map $f:\mathbb{A}^{\delta}\to\mathbb{B}^{\delta}$ is\footnote{Items 1 and 2 of Definition \ref{def:UC} are given in \cite{GeNaVe05} for arbitrary maps, but we restrict ourselves to the setting we need for the further development.

Items 3 and 4 are reported here as they appear in \cite{GeNaVe05}, and 5 appears in \cite{GeJo94} not as a definition, but as an equivalent characterization (cf.\ \cite[Corollary 3.6]{GeJo94}). 6 is the order dual version of 5.}:

\begin{enumerate}
\item \emph{upper continuous} (UC) if for any $u\in\mathbb{A}^{\delta}$ and any $q\in J^{\infty}(\mathbb{B}^{\delta})$, if $q\leq f(u)$ then $q\leq f(x)$ for some $x\in\kbbas$ s.t.\ $x\leq u$;
\item \emph{lower continuous} (LC) if for any $u\in\mathbb{A}^{\delta}$ and any $n\in M^{\infty}(\mathbb{B}^{\delta})$, if $n\geq f(u)$ then $n\geq f(y)$ for some $y\in\obbas$ s.t.\ $y\geq u$;
\item $f$ is \emph{Scott continuous}, if for all $u\in\mathbb{A}^{\delta}$ and all $q\in J^{\infty}(\mathbb{B}^{\delta})$, if $q\leq f(u)$ then there exists some $x\in J^{\infty}(\bbas)$ s.t.\ $x\leq u$ and $q\leq f(x)$;
\item $f$ is \emph{dually Scott continuous}, if for all $u\in\mathbb{A}^{\delta}$ and all $n\in M^{\infty}(\mathbb{B}^{\delta})$, if $n\geq f(u)$ then there exists some $y\in M^{\infty}(\bbas)$ s.t.\ $y\geq u$ and $n\geq f(y)$;
\item $f$ is \emph{continuous}, if for all $u\in\mathbb{A}^{\delta}$ and all $q\in J^{\infty}(\mathbb{B}^{\delta})$, if $q\leq f(u)$ then there exists some $x\in J^{\infty}_{\omega}(\bbas)$ s.t.\ $x\leq u$ and $q\leq f(x)$;
\item $f$ is \emph{dually continuous}, if for all $u\in\mathbb{A}^{\delta}$ and all $n\in M^{\infty}(\mathbb{B}^{\delta})$, if $n\geq f(u)$ then there exists some $y\in M^{\infty}_{\omega}(\bbas)$ s.t.\ $y\geq u$ and $n\geq f(y)$.
\end{enumerate}
\end{definition}

\begin{lemma}[\cite{GeJo04}, Theorem 2.15]\label{thm:largestUC:DML}
For any monotone map $f:\mathbb{A}\to\mathbb{B}$,
\begin{enumerate}
\item $f^{\sigma}:\mathbb{A}^{\delta}\to\mathbb{B}^{\delta}$ is the largest monotone UC extension of $f$ to $\mathbb{A}^{\delta}$;
\item $f^{\pi}:\mathbb{A}^{\delta}\to\mathbb{B}^{\delta}$ is the smallest monotone LC extension of $f$ to $\mathbb{A}^{\delta}$.
\end{enumerate}
\end{lemma}

Similar to the definitions of UC and LC, the following definitions can be given for monotone maps:

\begin{definition}
For any monotone map $f:\mathbb{A}^{\delta}\to\mathbb{B}^{\delta}$,
\begin{enumerate}
\item $f$ is \emph{strongly upper-continuous} (SUC) if for any $u\in\mathbb{A}^{\delta}$ and any $q\in K(\mathbb{B}^{\delta})$,
if $q\leq f(u)$ then $q\leq f(x)$ for some $x\in\kbbas$ s.t.\ $x\leq u$.
\item $f$ is \emph{strongly lower-continuous} (SLC) if for any $u\in\mathbb{A}^{\delta}$ and any $n\in O(\mathbb{B}^{\delta})$, if $n\geq f(u)$ then $n\geq f(y)$ for some $y\in\obbas$ s.t.\ $y\geq u$.
\end{enumerate}
\end{definition}

\noindent Next, we introduce a notion of continuity which is intermediate between Scott continuity and continuity. This notion does not appear in the relevant literature (cf.\ \cite{GeNaVe05}, \cite{GeJo94}), however, it is relevant for clarifying the nature of the small mistake in \cite{GeNaVe05} mentioned above. \begin{definition}
\label{def:wScoC}
If $m\geq 1$, and $f:\Pi_{i = 1}^m\mathbb{A}_i^{\delta}\to\mathbb{B}^{\delta}$ is monotone, then
\begin{enumerate}
\item $f$ is $m$-\emph{Scott continuous} if for any $\overline{u}\in\Pi_i\mathbb{A}_i^{\delta}$, and any $q\in J^{\infty}(\mathbb{B}^{\delta})$, if $q\leq f(\overline{u})$ then $q\leq f(\overline{x})$ for some $\overline{x}\in \Pi_i J^{\infty}(\bbas_i)$, s.t.\ $\overline{x} \leq \overline{u}$;
\item $f$ is \emph{dually $m$-Scott continuous} if for any $\overline{u}\in\Pi_i\mathbb{A}_i^{\delta}$, and any $n\in M^{\infty}(\mathbb{B}^{\delta})$, if $n\geq f(\overline{u})$ then $n\geq f(\overline{y})$ for some $\overline{y}\in \Pi_i M^{\infty}(\bbas_i)$, s.t.\ $\overline{u} \leq \overline{y}$.
\end{enumerate}
Scott and dual Scott continuity correspond to the definitions above for $m = 1$.
\end{definition}

\begin{remark}\label{rem:th 5.6 GNV is wrong}
\begin{enumerate}
\item For any $m\in\mathbb{N}$, the notion of $m$-Scott continuity is strictly stronger than continuity. Indeed, if $m=1$, then $m$-Scott continuity coincides with Scott continuity, which is stronger than continuity. If $m\geq 2$, then it can be easily verified that $$\Pi_{i=1}^{m}J^{\infty}(\A_{i}^{\delta})=\{\overline{j}\mid j_i\in J^{\infty}(\A_{i}^{\delta})\}\subsetneqq J^{\infty}_{\omega}(\Pi_{i=1}^{m}\A_{i}^{\delta}).$$
\item Notice also that, for $m\geq 2$, \[\Pi_i J^{\infty}(\bbas_i)\neq J^{\infty}(\Pi_i \bbas_i)\quad \mbox{ and }\quad \Pi_i M^{\infty}(\bbas_i)\neq M^{\infty}(\Pi_i \bbas_i).\]
Hence, being (dually) $m$-Scott continuous is different from being (dually) Scott continuous in the product. The difference between these notions is essentially the reason why  \cite[Theorem 5.6]{GeNaVe05} (stating that the $\sigma$-extension of any $n$-ary operator is Scott-continuous) is false for $n>1$.

Indeed, consider a binary operator $g: \A\times \A\to \B$ such that $g^{\sigma}(u_1, u_2)\neq\bot$, and let $q\in J^{\infty}(\B)$ such that $q\leq g^{\sigma}(u_1, u_2)$. Since any $\overline{p}\in J^{\infty}(\A^{\delta}\times\A^{\delta})$ is either of the form $(\bot, p_2)$ or of the form $(p_1, \bot)$ with $p_1, p_2\in J^{\infty}(\A^{\delta})$ and $g$ is normal by assumption, $g^{\sigma}(\overline{p})=\bot$, which contradicts the statement of \cite[Theorem 5.6]{GeNaVe05}.
\end{enumerate}
\end{remark}
The following lemma, which refines \cite[Lemma 4.2]{GeJo94} (cf.\ Remark \ref{rem:th 5.6 GNV is wrong}(i)), is the correct version of \cite[Theorem 5.6]{GeNaVe05}:

\begin{lemma}\label{lemma:sigmapi}
If $g : \Pi_{i = 1}^m \mathbb{B}_i\rightarrow \C$ is an operator, then $g^\sigma$ is $m$-Scott continuous.
\end{lemma}
\begin{proof}
Since $g$ is an operator, the map $g^{\sigma}$ is a complete operator (cf.\ Lemma \ref{lem:complete:additive}(iii)). Let $\overline{u}\in\Pi_i\mathbb{B}_i^{\delta}$, and $q\in J^{\infty}(\mathbb{C}^{\delta})$ s.t.\ $q\leq g^{\sigma}(\overline{u})$. As discussed on page \pageref{canext bdl is perfect}, each $\mathbb{B}_i^{\delta}$ is perfect for $1\leq i\leq m$. Therefore, for each $1\leq i\leq m$, we have that $u_i=\bigvee\mathcal{J}_i$ for some $\mathcal{J}_i\subseteq J^{\infty}(\B^{\delta}_i)$. Since $g^{\sigma}$ is a complete operator, $q\leq \bigvee\{g^{\sigma}(\overline{x})\mid \overline{x}\in\Pi_i\mathcal{J}_i\}$. As $q\in J^{\infty}(\mathbb{C}^{\delta})$, there exists some $\overline{x}\in\Pi_i\mathcal{J}_i\subseteq \Pi_iJ^{\infty}(\B^{\delta}_i)$ s.t.\ $\overline{x}\leq\overline{\bigvee\mathcal{J}_i}=\overline{u}$ and $q\leq g^{\sigma}(\overline{x})$. This shows that $g^{\sigma}$ is $m$-Scott continuous, as required.
\end{proof}

In \cite{GeNaVe05}, the following corollary is inferred from \cite[Theorem 5.6]{GeNaVe05}, which as discussed in Remark \ref{rem:th 5.6 GNV is wrong}(ii) is flawed. However, its statement holds as a consequence of Lemma \ref{lemma:sigmapi}. Its proof is straightforward, but it is reported below for the sake of giving an exhaustive exposition.
\begin{cor}[\cite{GeNaVe05}, Corollary 5.7]\label{cor:scott}

If $f_i : \mathbb{A}_i\rightarrow\mathbb{B}_i$ is monotone for any $1\leq i\leq m$ and $g: \Pi_{i = 1}^m \mathbb{B}_i\rightarrow\C$ is an operator, then $g^\sigma (f_1^\sigma, \ldots, f_m^\sigma) \leq (g(f_1, \ldots, f_m))^\sigma$.

\end{cor}
\begin{proof} Let $\overline{f}: \Pi\A_i\to \Pi\B_i$ be the product map of the maps $f_i$, and let $\overline{f^\sigma}$ be the product map of their $\sigma$-extensions.
Since $(g(\overline{f}))^\sigma$ is the greatest UC extension of the monotone map $g(\overline{f})$ (cf.\ Lemma \ref{thm:largestUC:DML}), it is enough to show that $g^\sigma (\overline{f^\sigma})$ is UC. Fix $\overline{u}\in \Pi\A_i$ and $q\in J^{\infty}(\C^{\delta})$ such that $q\leq g^\sigma (\overline{f^\sigma}(\overline{u}))$. By Lemma \ref{lemma:sigmapi}, this implies that $q\leq g^\sigma (\overline{x})$ for some $\overline{x}\in \Pi J^\infty(\B_i^\delta)$ such that $\overline{x}\leq \overline{f^\sigma}(\overline{u})$. Because (each component of) $\overline{f^\sigma}$ is UC, this implies that $\overline{x}\leq \overline{f^\sigma}(\overline{y})$ for some $\overline{y}\in \Pi K(\A_i^\delta) = K(\Pi\A_i^\delta)$ such that $\overline{y}\leq \overline{u}$. Finally, by monotonicity, $q\leq g^\sigma (\overline{x})\leq g^\sigma (\overline{f^\sigma}(\overline{y}))$, as required.
\end{proof}

\begin{remark}\label{rem:meet preserving not dual op}
In \cite{GeNaVe05}, in the discussion immediately above Theorem 5.8, it is written, about {\em meet-preserving} maps:
\begin{quote} Note that in the case of
unary operations, this [i.e.\ being meet preserving] is the same concept as being a dual operator; for operations of higher
rank, however, preserving meets is a far stronger condition.
\end{quote}
Actually, for non unary operations, being meet-preserving $($resp.\ join-preserving$)$ and being a dual operator $($resp.\ an operator$)$ are {\em incomparable} notions. For instance, in any lattice $\mathbb{L}$, the binary meet operator $\wedge:\mathbb{L}\times\mathbb{L}\to\mathbb{L}$, mapping any $(a, b)\in \mathbb{L}\times\mathbb{L}$ to the greatest lower bound of $\{a, b\}$, is by definition the right adjoint of the diagonal map $\Delta: \mathbb{L}\to\mathbb{L}\times\mathbb{L}$, mapping any $a$ to $(a, a)$. Hence, $\wedge:\mathbb{L}\times\mathbb{L}\to\mathbb{L}$ preserves existing meets (and in particular all finite ones) in $\mathbb{L}\times\mathbb{L}$. However, the operator $\wedge$ is not a dual operator: indeed, for any $a\neq \top$ we have $a\wedge \top = a \neq \top$. Likewise, $\vee:\mathbb{L}\times\mathbb{L}\to\mathbb{L}$ is left adjoint to $\Delta$ and hence is join-preserving, but not an operator. However, it is not difficult to show that any non-unary meet-preserving $($resp.\ join-preserving$)$ map is {\em multiplicative} $($resp.\ {\em additive}$)$. This is essentially because, for all lattices $\A_i$, $1\leq i\leq n$, any $\overline{a}\in \Pi\A_i$ and any $S\subseteq \A_j$ for a given $1\leq j\leq n$, the equality \[\overline{a}[(\bigwedge S)/a_j] = \bigwedge_{\Pi\A_i} \{\overline{a}[s/a_j]\mid s\in S\}\] holds if $S\neq \varnothing$ $($but it is not guaranteed to hold if $S = \varnothing$, as the example of $\wedge$ shows$)$.

As will be discussed in Section \ref{sec:weaker sufficient conditions}, the observation that non-unary meet-preserving $($resp.\ join-preserving$)$ maps are multiplicative $($resp.\ additive$)$ will be important for remedying to another, slightly more serious issue in \cite{GeNaVe05}, the solution of which is at the base of the generalization of Section \ref{sec:5}.
\end{remark}

\begin{lemma}[\cite{GeNaVe05}, Theorem 5.8]\label{lem:meet-pres_suc}
If $f:\A\rightarrow\B$ is meet-preserving, then $f^\sigma$ is SUC.
\end{lemma}

\begin{cor}[\cite{GeNaVe05}, Corollary 5.9]\label{cor:suc:GNV}

If $f:\A\rightarrow\B$ is meet-preserving and $g:\mathbb{B}\rightarrow\C$ is monotone, then
$g^\sigma f^\sigma\leq(gf )^\sigma$.
\end{cor}

Corollaries \ref{cor:scott} and \ref{cor:suc:GNV} above constitute the core of the J\'onsson-style canonicity argument of \cite{GeNaVe05}. We will discuss them in Section \ref{sec:jonsson strategy}.

\subsection{Distributive lattice expansions}\label{subsec:dist:latt:expansions}

Distributive modal logic has been introduced in \cite{GeNaVe05}, and further studied in \cite{ConPal12}.
Algebraically, distributive modal logic is interpreted on certain BDLs expanded with the unary operations $\Diamond, {\lhd}, \Box, {\rhd}$ such that $\Diamond$ and ${\lhd}$ are operator-type, in that they respectively preserve joins and reverse meets, and $\Box$ and ${\rhd}$ are dual operator-type, in that they respectively preserve meets and reverse joins. In what follows, we report on a more compact variant of the DML setting, in which the operator-type connective $f$ plays the role of $\Diamond, {\lhd}$, and the dual operator-type connective $g$ plays the role of $\Box, {\rhd}$.

For any order-type $\eta\in \{1,\partial\}^2$, and given a countable set $\mathsf{Prop}$ of proposition variables, the language of DLE$_{\eta}$ is defined recursively as follows:
$$\varphi::=p \in \mathsf{Prop}\mid \top\mid\ \perp\ \mid\varphi\land\varphi\mid\varphi\lor\varphi\mid f(\varphi)\mid g(\varphi).$$
We will not expand on the axiomatization of DLE$_{\eta}$ and refer the reader to \cite{GeNaVe05} for a complete treatment. Here we only assume that, for any $\eta$, the axiomatics of DLE$_{\eta}$ requires that the algebraic interpretation of $f$ and $g$ is such that the product map $(f^\A, g^\A): \A^\eta\rightarrow \A\times \A^\partial$ preserves finite joins in each coordinate. For instance, if $\eta = (1, \partial)$, then $f$ is $\Diamond$-type and $g$ is $\Box$-type; if $\eta = (\partial, 1)$, then $f$ is ${\lhd}$-type and $g$ is ${\rhd}$-type, and so on.

In what follows, we will abuse terminology and refer to the algebras naturally associated with the logic DLE$_\eta$ as DLEs. We will further disambiguate only when necessary. Moreover, for the sake of not overloading notation, we will omit the subscript $\eta$ unless strictly necessary. Since each unary operation $f$ and $g$ in a DLE is either join- (resp.\ meet-)preserving or reversing, each of them is smooth (cf.\ \cite{GeNaVe05}, discussion above Definition 2.19). Therefore, the canonical extension of a DLE can be defined as follows (cf.\ \cite{GeNaVe05}, Definitions 2.19 and 2.20):

\begin{definition}
For any DLE $\mathbb{A}=(A, f, g$), the canonical extension $\mathbb{A}^{\delta}$ of $\bba$ is defined as $\mathbb{A}^{\delta}=(A, f^{\sigma}, g^{\sigma}) = (A, f^{\pi}, g^{\pi})$.
\end{definition}

Clearly, every DLE$_{\eta}$-term $\alpha$ is interpreted in $\ca$ as the appropriate composition of operations of $\ca$ that interpret the logical connectives, and we denote this term function as $\alpha^{\ca}$. On the other hand, the term function $\alpha^{\A}$ can be extended to $\ca$ via its $\sigma$- or $\pi$-extension. The following definition (cf.\ \cite[Definition 5.2]{GeNaVe05}) compare these functions:

\begin{definition}\label{def:stable}
 A DLE$_{\eta}$-term $\alpha$ is, respectively, \emph{$\lambda$-stable}, \emph{$\lambda$-expanding}, and \emph{$\lambda$-contracting} $(\lambda\in\{\sigma, \pi\})$ if the following respectively hold for any DLE $\A$: \[\alpha^{\A^\delta} = (\alpha^\A)^\lambda,\quad\quad \alpha^{\A^\delta} \leq (\alpha ^\A)^\lambda,\quad\quad \alpha^{\A^\delta}\geq(\alpha^\A)^\lambda.\]
\end{definition}

\subsection{Sahlqvist DLE$_{\eta}$- and DML-inequalities}\label{subsec:DLE:eta}

Any DLE$_{\eta}$-term $s$ can be associated with its positive (resp.\ negative) \emph{signed generation tree}, by labelling the root of the generation tree of $s$ with $+$ (resp.\ $-$) and then propagating the label along the tree according to the polarity of each node: the label is kept the same along positive polarities and is switched along negative ones.
For any term $s(p_1,\ldots p_n)$, any order type $\epsilon$ over $n$, and any $1 \leq i \leq n$, an \emph{$\epsilon$-critical node} in a signed generation tree of $s$ is a leaf node $+p_i$ with $\epsilon_i = 1$ or $-p_i$ with $\epsilon_i = \partial$. An $\epsilon$-{\em critical branch} in the tree is a branch ending in an $\epsilon$-critical node.

For every term $s(p_1,\ldots p_n)$ and every order type $\epsilon$, we say that $\ast s$ ($\ast\in \{+, -\}$) is \emph{$\epsilon$-uniform}, or that $\ast s$ {\em agrees with} $\epsilon$, if every leaf in the signed generation tree of $\ast s$ is $\epsilon$-critical. A signed term $\ast s$ is \emph{uniform} if it is $\epsilon$-uniform for some order-type $\epsilon$.
\begin{table}[\here]
\begin{center}
\begin{tabular}{| c | c |}
\hline
 Skeleton &PIA\\
\hline
$\Delta$-adjoints & SRA \\

\begin{tabular}{ c c c c c}
$+$ &$\vee$ &$\wedge$ &$\phantom{\cdot}$ & \\
$-$ &$\wedge$ &$\vee$\\
\hline
\end{tabular}
&
\begin{tabular}{c c c c}
$+$ &$g$ &$\wedge$ \\
$-$ &$f$ &$\vee$ \\
\hline
\end{tabular}
\\
SLR &SRR\\

\begin{tabular}{c c c c }
$+$ & $\wedge$ & $f$ \\
$-$ & $\vee$ &$g$ \\
\end{tabular}
&
\begin{tabular}{c c c c}
$+$ &$\vee$ &$\phantom{\rightarrow}$& $\phantom{\circ}$\\
$-$ & $\wedge$ \\
\end{tabular}
\\
\hline
\end{tabular}
\end{center}
\caption{Classification of nodes for the signature of DLE$_\eta$}\label{Table:classification of DLEeta}
\end{table}

\begin{definition}[cf.\ \cite{CoGhPa14}, Definition 36.4]\label{Excellent:Branch:Def}
Nodes in signed generation trees are classified according to Table \ref{Table:classification of DLEeta}. For $*\in\{+,-\}$, a branch in a signed generation tree $\ast s$ is:
\begin{itemize}
\item a \emph{good branch} if it is the concatenation of two paths $P_1$ and $P_2$, one of which may possibly be of length $0$, such that $P_1$ is a path from the leaf consisting (apart from variable nodes) only of PIA-nodes, and $P_2$ consists (apart from variable nodes) only of Skeleton-nodes.

\item an \emph{excellent branch} if it is good, and moreover $P_1$ consists (apart from variable nodes) only of SRA-nodes.
\end{itemize}
\end{definition}
\begin{definition}\label{Sahlqvist:Ineq:Def}
For any $n$-order type $\epsilon$, the signed generation tree $\ast s$ of a term $s(p_1,\ldots p_n)$ is \emph{$\epsilon$-Sahlqvist} if for all $1 \leq i \leq n$, every $\epsilon$-critical branch with leaf labelled $p_i$ is excellent.

An inequality $s \leq t$ is \emph{$\epsilon$-Sahlqvist} if the trees $+s$ and $-t$ are both $\epsilon$-Sahlqvist. An inequality $s \leq t$ is \emph{Sahlqvist} if it is $\epsilon$-Sahlqvist for some $\epsilon$.
\end{definition}

\begin{remark}
Since the signature of DLE$_{\eta}$ is a reduct of DML, Definition \ref{Sahlqvist:Ineq:Def} could have been given by suitably restricting the corresponding definitions in \cite{GeNaVe05} and \cite{ConPal12}. However, the definition above rather follows \cite{CoGhPa14}. As discussed in \cite{CoGhPa14}, Sahlqvist inequalities are defined purely in terms of the order-theoretic properties of the algebraic interpretations of the logical connectives. 
 These order-theoretic properties remain essentially unchanged when applied to different signatures and different logics. Hence in particular, the definition of Sahlqvist DML-inequalities given in \cite{GeNaVe05} can be equivalently given by applying Definition \ref{Sahlqvist:Ineq:Def} verbatim on the classification of nodes reported in Table \ref{Table:classification of DML} below.
\end{remark}

\begin{table}[\here]
\begin{center}
\begin{tabular}{| c | c |}
\hline
 Skeleton &PIA\\
\hline
$\Delta$-adjoints & SRA \\

\begin{tabular}{ c c c c c}
$+$ &$\vee$ &$\wedge$ &$\phantom{\cdot}$ & \\
$-$ &$\wedge$ &$\vee$\\
\hline
\end{tabular}
&
\begin{tabular}{c c c c}
$+$ &$\Box$ & ${\rhd}$ & $\wedge$ \\
$-$ &$\Diamond$ & ${\lhd}$ &$\vee$ \\
\hline
\end{tabular}
\\
SLR &SRR\\

\begin{tabular}{c c c c }
$+$ & $\wedge$ & $\Diamond$ & ${\lhd}$ \\
$-$ & $\vee$ &$\Box$ & ${\rhd}$ \\
\end{tabular}
&
\begin{tabular}{c c c c}
$+$ &$\vee$ &$\phantom{\rightarrow}$& $\phantom{\circ}$\\
$-$ & $\wedge$ \\
\end{tabular}
\\
\hline
\end{tabular}
\end{center}
\caption{Classification of nodes for the signature of DML}\label{Table:classification of DML}
\end{table}

\begin{remark}\label{rem:pia}{\em
Let us expand on the criteria motivating the classification of nodes in the Table \ref{Join:and:Meet:Friendly:Table:HAR:DLR}. Firstly, we are using two types of names: ``meaningful'' names such as {\em syntactically right residual}, and more ``cryptic'' names, such as Skeleton and PIA. Hence, the resulting classification has two layers, one of which accounts for what the connectives are order theoretically, and the other for what they are for, in the context of ALBA. Specifically, the ``meaningful'' names explicitly refer to intrinsic order-theoretic properties of the interpretation of the logical connectives. For example, the SAC connectives will be interpreted as operations which are additive in each coordinate (modulo order-type). The more ``cryptic'' names classify connectives in terms of the kind of rules which will be applied to them by ALBA. The idea is that only approximation/splitting rules are to be applied to skeleton nodes, with the aim of   surfacing the PIA subterms containing the $\epsilon$-critical occurrences of propositional variables, and only residuation/adjunction rules are to be applied to PIA nodes, with the aim of computing the ``minimal valuations''. The fact that these rules can be soundly applied is guaranteed by the intrinsic order-theoretic properties. Sometimes, the intrinsic order-theoretic properties of a connective are such that other rules are also soundly applicable to it. However, the Skeleton/PIA classification indicates that this is not required in order to reach Ackermann shape. The term Skeleton formula comes from \cite{BenthemBH12}. The acronym PIA was introduced by van Benthem in \cite{Benthem05}. The analysis of PIA-formulas conducted in \cite{Benthem05,BenthemBH12}  can be summarized in the slogan ``PIA formulas provide minimal valuations'', which  is precisely the role of those terms which we call PIA-terms here. Again, this choice of terminology is not based on the original syntactic description of van Benthem, but rather on which rules are best being applied to them in order to be guaranteed success  of the execution of the algorithm.  In this respect, the crucial property possessed by PIA-formulas in the setting of normal modal logic is the \emph{intersection property}, isolated by van Benthem in \cite{Benthem05}, which is enjoyed by those formulas which, seen as operations on the complex algebra of a frame, preserve  arbitrary intersections of subsets. The order-theoretic import of this property is clear: a formula has the intersection property iff the term function associated with it is completely meet-preserving. In the complete lattice setting, this is equivalent to it being a right adjoint; this is exactly the order-theoretic property guaranteeing the soundness of adjunction/residuation rules in the setting of normal modal logic.}
\end{remark}
\section{J\'onsson's strategy for canonicity}
\label{sec:jonsson strategy}
The present section is aimed at illustrating J\'onsson's methodology. We first sketch the proof in \cite{GeNaVe05} of the canonicity of Sahlqvist inequalities, and then discuss J\'onsson's strategy.

\subsection{J\'onsson-style canonicity for Sahlqvist inequalities}

\begin{theorem}[cf.\ \cite{GeNaVe05}, Theorem 5.1]\label{thm:canonicity:Sahlqvist}
For any Sahlqvist DML-inequality $\alpha\leq\beta$ and any DMA $\mathbb{A}$, \[\mathbb{A}\models \alpha\leq\beta\Rightarrow\mathbb{A}^\delta\models \alpha\leq\beta.\]
\end{theorem}

\begin{proof}
The proof consists of the following implications:

\vspace{2mm}
\begin{center}
\begin{tabular}{lll}
\vspace{2mm}
& $\A\models \alpha \leq \beta $\\
\vspace{2mm}
$\Longleftrightarrow$& $\alpha^\A \leq \beta^\A$ &(by definition)\\
\vspace{2mm}
$\Longleftrightarrow$ &$\alpha_1^\A\leq \beta_1^\A\lor \gamma^\A$ & \cite[Lemma 5.14]{GeNaVe05}\\
\vspace{2mm}
$\Longrightarrow$ &$\left(\alpha_1^\A\right)^\sigma\leq \left(\beta_1^\A\right)^\pi\lor\left(\gamma^\A\right)^\sigma$ &\cite[Lemma 5.11]{GeNaVe05}\\
\vspace{2mm}
$\Longrightarrow$ & $\alpha_1^{\A^\delta}\leq\beta_1^{\A^\delta}\lor\gamma^{\A^\delta}$ &\cite[Lemmas 5.5 and 5.10]{GeNaVe05}\\
\vspace{2mm}
$\Longleftrightarrow$ & $\alpha^{\A^\delta}\leq\beta^{\A^\delta}$ & \cite[Lemma 5.14]{GeNaVe05}\\
\vspace{2mm}
$\Longleftrightarrow$ & $\A^\delta\models\alpha\leq\beta$ &(by definition).
\end{tabular}
\end{center}
\end{proof}

There are two key steps to the proof of Theorem \ref{thm:canonicity:Sahlqvist}. The first step equivalently transforms the inequality $\alpha\leq\beta$ into some inequality $\alpha_1\leq\beta_1\lor\gamma$ s.t.\ $\alpha_1\leq \beta_1$ is a uniform Sahlqvist inequality, $\gamma$ is a uniform term, and certain additional conditions are satisfied. This step is not involved in the refinement of the present paper, and for a discussion on it, the reader is referred to the companion paper \cite{PaSoZh14} in which this step is discussed, refined, and generalized.\footnote{The refinement in \cite{PaSoZh14} consists in an alternative proof which does not make use of \cite[Lemma 5.11]{GeNaVe05}. The generalization consists in the proof of canonicity of the inequalities in the language of DML on which the algorithm ALBA succeeds in calculating a first-order correspondent. This class is a significant proper extension of the Sahlqvist class.} The second step consists of showing that if $\alpha_1\leq \beta_1$ is uniform Sahlqvist, then $\alpha_1$ is $\sigma$-expanding and $\beta_1$ is $\pi$-contracting (cf.\ \cite[Lemma 5.10]{GeNaVe05}), and that if $\gamma$ is uniform, then $\gamma$ is $\sigma$-contracting (cf.\ \cite[Lemma 5.5]{GeNaVe05}).
 In the following subsection, we give a closer look at the second step, and among other things, we discuss a mistake in the proof of \cite[Lemma 5.10]{GeNaVe05} which reduces the scope of the result in \cite{GeNaVe05} to a proper fragment of Sahlqvist inequalities.

\subsection{Properties of Sahlqvist terms}\label{subsec:3:1:2}

The core of J\'onsson's method as applied in \cite{GeNaVe05}, and the step that we are going to generalize in the following section, is the proof that if $\alpha_1\leq \beta_1$ is uniform Sahlqvist, then $\alpha_1$ is $\sigma$-expanding and $\beta_1$ is $\pi$-contracting. We only consider the case of $\alpha_1$ in the present subsection, the case of $\beta_1$ being order-dual.
Recall that the $\sigma$-extension of a monotone map is its greatest UC extension (cf.\ Lemma \ref{thm:largestUC:DML}). Hence, in order to show that a term $t$ is $\sigma$-expanding, i.e.\ that $t^{\ca}\leq(t^{\A})^{\sigma}$, it suffices to show that $t^{\ca}$ is UC.
The term function $t^{\ca}$ is the composition of the $\sigma$-extensions of the interpretations of the logical connectives occurring in $t$.

Let us start by considering the composition of two monotone maps $f: \A\to \B$ and $g:\B\to \C$, in search for conditions which guarantee the composition $g^{\sigma}f^{\sigma}$ to be UC.

Recall that upper continuity, Scott-continuity, and strong upper-continuity are similar conditions, and specifically, their quantification patterns are of the following forms, respectively:

\vspace{2mm}

\begin{tabular}{l r}
\vspace{2mm}
\mbox{for all $q\in J^{\infty}(\B^{\delta})\ldots$ there exists some $x\in\kbbas$ \ldots} & ($\forall J\exists K$)\\
\vspace{2mm}
\mbox{for all $q\in J^{\infty}(\B^{\delta})\ldots$ there exists some $x\in J^{\infty}(\A^{\delta})$ \ldots} & ($\forall J\exists J$)\\
\vspace{2mm}
\mbox{for all $q\in K(\B^{\delta})\ldots$ there exists some $x\in K(\A^{\delta})$ \ldots}
 & ($\forall K\exists K$)\\
\end{tabular}

This quantification pattern suggests that there are two immediately available sufficient conditions on $f$ and $g$ which guarantee $g^{\sigma}f^{\sigma}$ to be UC:

\begin{itemize}

\item either $g^{\sigma}$ is Scott continuous ($\forall J\exists J$) and $f^{\sigma}$ is UC ($\forall J\exists K$);

\item or $g^{\sigma}$ is UC ($\forall J\exists K$) and $f^{\sigma}$ is SUC ($\forall K\exists K$).

\end{itemize}

Since $\sigma$-extensions of monotone functions are always UC, only half of each condition needs to be guaranteed: namely, in the first case $g^{\sigma}$ needs to be Scott continuous (and actually, $m$-Scott continuity is enough), and in the second case $f^{\sigma}$ needs to be SUC.

Corollaries \ref{cor:scott} and \ref{cor:suc:GNV} are indeed proved by making use of (a refinement of) the first and of the second option, respectively.
Specifically, as discussed in subsection \ref{ssec:continuity properties}, in \cite{GeNaVe05}, the map $g$ being an operator was taken as a sufficient condition for the first case to apply (Corollary \ref{cor:scott}), and $f$ being meet-preserving was taken as a sufficient condition for the second case to apply (Corollary \ref{cor:suc:GNV}).
Based on these two corollaries, \cite[Lemma 5.10]{GeNaVe05} claims that if $\alpha_1\leq \beta_1$ is uniform Sahlqvist, then $\alpha_1$ is $\sigma$-expanding and $\beta_1$ is $\pi$-contracting.

Hence, conceptually, \cite[Lemma 5.10]{GeNaVe05} motivates the syntactic shape of uniform Sahlqvist inequalities in terms of the order-theoretic behaviour of the interpretation of the logical connectives, and of the properties of their resulting composition.

The proof that $\alpha_1$ is $\sigma$-expanding is done by induction on $\alpha_1$. The cases in which the main connective of $\alpha_1$ is $f\in \{\Diamond, \wedge, \vee\}$ are treated simultaneously, ``as $f$ is an operator in all these cases.'' However, as discussed in Remark \ref{rem:meet preserving not dual op}, the connective $\vee$ is not an operator, and hence Corollary \ref{cor:scott} does not apply to it.

In conclusion, the proof of \cite[Lemma 5.10]{GeNaVe05} (hence that of \cite[Theorem 5.1]{GeNaVe05}, cf.\ Theorem \ref{thm:canonicity:Sahlqvist}) is incomplete, and the part of it which is proven only accounts for the canonicity of the fragment of $\epsilon$-Sahlqvist inequalities such that any critical branch is the concatenation of two paths $P_1$ and $P_2$, one of which may possibly be of length $0$, such that $P_1$ is a path from the leaf consisting (apart from variable nodes) only of SRA-nodes, and $P_2$ consists (apart from variable nodes) only of SLR-nodes.

\subsection{Weaker sufficient conditions for J\'onsson's argument}
\label{sec:weaker sufficient conditions}
In the present section, we weaken the sufficient conditions for the J\'onsson-style argument which feature in Corollaries \ref{cor:scott} and \ref{cor:suc:GNV}.

As we saw in the previous section, requiring the outer maps to be operators does not account for the full class of Sahlqvist inequalities, due to the presence of the nodes $+\vee$ and $-\wedge$, which are left adjoints but not operators.

However, as observed in Remark \ref{rem:meet preserving not dual op}, left adjoints (of any arity) preserve non-empty joins in each coordinate. That is, they are {\em additive} (cf.\ Definition \ref{def:order theoretic properties}).
Additivity is the weaker condition which is shown below to be sufficient for the J\'onsson-style argument. The basic idea is that outer maps are required to be {\em weakly $m$-Scott continuous} (see Definition \ref{def:weak:m:Scott} below). In what follows, we will find it useful to let the symbols $J^{\infty}_{\bot}(\mathbb{A}^{\delta})$ and $M^{\infty}_{\top}(\mathbb{A}^{\delta})$ abbreviate the sets $J^{\infty}(\mathbb{A}^{\delta})\cup \{\bot\}$ and $M^{\infty}(\mathbb{A}^{\delta})\cup \{\top\}$ respectively, for any DLE $\mathbb{A}$. For the sake of readability, $\Pi $ will abbreviate $\Pi_{i = 1}^m$.

\begin{definition}\label{def:weak:m:Scott}
For any monotone map $f:\Pi\mathbb{A}_i^{\delta}\to\mathbb{B}^{\delta}$,
\begin{enumerate}
\item[$(\forall J \exists J^m_\bot)$] $f$ is \emph{weakly $m$-Scott continuous} if for any $\overline{u}\in \Pi\mathbb{A}_i^{\delta}$ and any $q\in J^{\infty}(\mathbb{B}^{\delta})$,
if $q\leq f(\overline{u})$ then $q\leq f(\overline{x})$ for some $\overline{x}\in \Pi J^{\infty}_{\bot}(\bbas_i)$ s.t.\ $\overline{x}\leq \overline{u}$.
\item[$(\forall M \exists M^m_\top)$] $f$ is \emph{dually weakly $m$-Scott continuous} if for any $\overline{u}\in \Pi\mathbb{A}_i^{\delta}$ and any $q\in M^{\infty}(\mathbb{B}^{\delta})$,
if $q\geq f(\overline{u})$ then $q\geq f(\overline{x})$ for some $\overline{x}\in \Pi M^{\infty}_{\top}(\bbas_i)$ s.t.\ $\overline{x}\geq \overline{u}$.
\end{enumerate}
\end{definition}

\begin{lemma}\label{thm:operator:Scott}
For any monotone map $f:\Pi\A^{\delta}_i\to\B^{\delta}$,
\begin{enumerate}
\item if $f$ is completely additive, then $f$ is weakly $m$-Scott continuous;
\item if $f$ is completely multiplicative, then $f$ is dually weakly $m$-Scott continuous.
\end{enumerate}
\end{lemma}

\begin{proof}
(i): Let $\overline{u}\in\Pi \mathbb{A}^{\delta}_i$ and $q\in J^{\infty}(\mathbb{B}^{\delta})$ such that $q\leq f(\overline{u})$. As discussed on page \pageref{canext bdl is perfect}, $\mathbb{A}_i^{\delta}$ is perfect for $1\leq i\leq m$, therefore $u_i=\bigvee\mathcal{J}_i$ for some $\mathcal{J}_i\subseteq J^{\infty}(\ca_i)$.
Since $f$ is completely additive, $q\leq\bigvee\{f(\overline{x})\mid \overline{x}\in\Pi(\mathcal{J}_i\cup\{\bot_i\})\}$. Since $q\in J^{\infty}(\mathbb{B}^{\delta})$, this implies that $q\leq f(\overline{x})$ for some $\overline{x}\in\Pi(\mathcal{J}_i\cup\{\bot_i\})\subseteq \Pi J^{\infty}_\bot(\bbas_i)$ s.t.\ $\overline{x}\leq(\bigvee\mathcal{J}_i)_{i = 1}^m=\overline{u}$. Item (ii) is shown order dually.
\end{proof}

\begin{corollary}\label{lem:gen:Scott}
For all monotone maps $f_i:\A_i\to{\B}_i$, $1\leq i\leq m$ and $g:\Pi{\B}_i\to\C$,
\begin{enumerate}
\item if $g$ is additive, then $g^{\sigma}(f_1^{\sigma}, \ldots, f_m^{\sigma})\leq(g(f_1, \ldots, f_m))^{\sigma}$;
\item if $g$ is multiplicative, then $g^{\pi}(f_1^{\pi}, \ldots, f_m^{\pi})\geq(g(f_1, \ldots, f_m))^{\pi}$.
\end{enumerate}
\end{corollary}

\begin{proof}
(i) Let $\overline{f}: \Pi\A_i\to \Pi\B_i$ be the product map of the maps $f_i$, and let $\overline{f^\sigma}$ be the product map of their $\sigma$-extensions.

Since $(g(\overline{f}))^\sigma$ is the greatest UC extension of the monotone map $g(\overline{f})$ (cf.\ Lemma \ref{thm:largestUC:DML}), it is enough to show that $g^\sigma (\overline{f^\sigma})$ is UC. Fix $\overline{u}\in \Pi\A_i$ and $q\in J^{\infty}(\C^{\delta})$ such that $q\leq g^\sigma (\overline{f^\sigma}(\overline{u}))$. By Lemmas \ref{lem:complete:additive} and \ref{thm:operator:Scott}, this implies that $q\leq g^\sigma (\overline{x})$ for some $\overline{x}\in \Pi J^{\infty}_\bot(\mathbb{B}_i^{\delta})$ such that $\overline{x}\leq \overline{f^\sigma}(\overline{u})$. Because (each component of) $\overline{f^\sigma}$ is UC, if $x_i\neq\bot_i$, then $x_i\leq f_i^{\sigma}(x_i)$ for some $y_i\in K(\ca)$ s.t.\ $y_i\leq u_i$; if $x_i=\bot_i$, then letting $y_i = \bot_i\in K(\ca)$ satisfies both $x_i\leq f_i^{\sigma}(y_i)$ and $v_i\leq u_i$.
In either case, $\overline{x}\leq \overline{f^\sigma}(\overline{y})$ for some $\overline{y}\in \Pi K(\A_i^\delta) = K(\Pi\A_i^\delta)$ such that $\overline{y}\leq \overline{u}$. Finally, by monotonicity, $q\leq g^\sigma (\overline{x})\leq g^\sigma (\overline{f^\sigma}(\overline{y}))$, as required. Item (ii) is shown order dually.
\end{proof}
\begin{remark}
The proof of \cite[Lemma 5.10]{GeNaVe05} can be emended by treating $\vee$ separately from $f\in \{\Diamond, \wedge\}$, and appealing to Corollary \ref{lem:gen:Scott} above in the case of the induction step in which $\vee$ is the main connective. This completes the proof of Theorem \ref{thm:canonicity:Sahlqvist}.
\end{remark}

\section{Canonicity in a regular setting}\label{sec:5}
As we have seen in the previous section, additivity, rather than being an operator, is the key notion for proving the canonicity of the full class of Sahlqvist inequalities via J\'onsson's argument. This also implies that J\'onsson's argument can go through in settings in which modal connectives are interpreted algebraically as additive maps rather than operators. This motivates the treatment of the present section, in which the J\'onsson-style argument for canonicity is extended to the setting of regular distributive lattice expansions (DLRs). In what follows, when additive and multiplicative maps are mentioned, the order-type with respect to which the maps are additive or multiplicative is omitted.

\subsection{Regular distributive lattice expansions}
\label{ssec:DLR}
\begin{definition}\label{def:dlr}
A \emph{regular distributive lattice expansion} (DLR) is a structure $\A=(A, f, g, k, l)$, such that $A$ is a bounded distributive lattice, $f$ is a unary additive map, $g$ is a unary multiplicative map, $k$ is a $m$-ary additive map and $l$ is an $n$-ary multiplicative map. \end{definition}
\begin{remark} The algebraic signature of the definition above reflects the logical signature which we are going to consider, about which
some observations are in order.
\begin{enumerate}
\item The choice of the signature above is motivated by the need to highlight the difference between the order-theoretic properties of the unary operators and the non-unary ones. Indeed, when $n>1$, the connectives $k$ and $l$ cannot be assumed to be smooth, and then their stability will play a crucial
 role.
 \item The DLR signature samples different order-theoretic behaviours. Hence, it can be used as as a template, in which different situations can be accounted for. In particular, the theory developed for this signature can be straightforwardly extended to signatures consisting of various copies (e.g.\ $f_1, f_2\ldots$) of the same connective (e.g.\ $f$) having the same order-theoretic behaviour of the `template' one. The DML signature and the signature of substructural logics can be obtained as special cases of the DLR signature. For instance, the results in \cite{GeNaVe05} are obtained as a special case, and (the ``regular'' versions of) substructural logics are encompassed.
\end{enumerate}
\end{remark}
\begin{lemma}[\cite{HBMoL}, Proposition 111.3]\label{lem:smooth}

If $f : \mathbb{A}\rightarrow\mathbb{B}$ is additive or multiplicative\footnote{If $\A = \Pi_i\A_i$, then the assumption would be equivalently restated in terms of $f$ being p-additive, or p-multiplicative.} (cf.\ Definition \ref{def:order theoretic properties}), then $f$ is smooth.
\end{lemma}

\begin{definition}\label{def:canextdlr}
 The canonical extension of a DLR $\A=(A, f, g, k, l)$ is defined as the tuple $\A^\delta=(A^\delta, f^\sigma, g^\sigma, k^\sigma,l^\pi)$, where $A^\delta$ is the canonical extension of the BDL $A$ and $f^\sigma, g^\sigma, k^\sigma,l^\pi$ are defined according to Definition \ref{def:canonical:extension:maps}.
 \end{definition}
\begin{remark} \begin{enumerate}
\item Unlike $f$ and $g$, the maps $k$ and $l$ might be non-smooth (cf.\ \cite{HBMoL}, Example 110), and hence in defining the canonical extension of a DLR, we need to choose which of their extensions to take. Our choice in Definition \ref{def:canextdlr} is motivated by the fact that the $\sigma$-extension of an (p-)additive map is completely (p-)additive and the $\pi$-extension of a (p-)multiplicative map is completely (p-)multiplicative (cf.\ Lemma \ref{lem:complete:additive}). So, Definition \ref{def:canextdlr} guarantees that the canonical extension of a DLR is a perfect DLR.\footnote{A \emph{perfect DLR} is a DLR whose underlying BDL is perfect and its additive maps (resp. multiplicative maps) preserve arbitrary non-empty joins (resp. arbitrary non-empty meets).}.

 \item Moreover, by definition, $k$ is $\sigma$-stable and $l$ is $\pi$-stable (cf.\ Definition \ref{def:stable}). As we will see in the proof of Lemma \ref{lem:additive:difficult}, stability is enough to guarantee the analogue of \cite[Lemma 5.10]{GeNaVe05} to go through.
\end{enumerate}
\end{remark}
\noindent For a set $\mathsf{Prop}$ of propositional variables, the language DLR of regular distributive lattices expansions is defined recursively as follows:
$$\varphi::=p \in \mathsf{Prop}\mid \top\mid\ \perp\ \mid\varphi\land\varphi\mid\varphi\lor\varphi\mid f(\varphi)\mid g(\varphi)\mid k(\overline{\varphi})\mid l(\overline{\varphi}).$$\label{def:regular:modal:logic}
The nodes in the signed generation tree of DLR-terms are classified according to the table below:

\begin{table}[\here]
\begin{center}
\begin{tabular}{| c | c|}
\hline
SAC & SMP \\
\hline
\begin{tabular}{ c c c c c}
$+$ &$\vee$ &$\wedge$ &$f$ & $k$ \\
$-$ &$\wedge$ & $\vee$ & $g$ & $l$\\

\end{tabular}
&
\begin{tabular}{ c c c}
$+$ &$\wedge$&$g$\\
$-$ &$\vee$ & $f$\\

\end{tabular}\\
\hline
\end{tabular}
\end{center}
\caption{SAC and SMP nodes for $\mathrm{DLR}$. }
\label{Join:and:Meet:Friendly:Table}
\end{table}

\begin{definition}\label{Def:Good:Branches}
Nodes in signed generation trees will be called \emph{syntactically additive coordinate-wise (SAC)}, \emph{syntactically p-multiplicative (SMP)}, according to the specification given in Table \ref{Join:and:Meet:Friendly:Table}. A branch in a signed generation tree $\ast s$, with $\ast \in \{+, - \}$, is called an \emph{excellent branch} if it is the concatenation of two paths $P_1$ and $P_2$, one of which may possibly be of length $0$, such that $P_1$ is a path from the leaf consisting (apart from variable nodes) only of SMP-nodes, and $P_2$ consists (apart from variable nodes) only of SAC-nodes.
\end{definition}

\begin{definition}\label{def:rsahlqvist}Given an order type $\epsilon$, the signed generation tree $\ast s$ (for $\ast \in \{-, + \}$) of a term $s(p_1,\ldots p_n)$ is \emph{$\epsilon$-regular Sahlqvist} ($\epsilon$-DLR-Sahlqvist)
if for all $1 \leq i \leq n$, every $\epsilon$-critical branch with leaf labelled $p_i$ is excellent.

An inequality $s \leq t$ is \emph{$\epsilon$-regular Sahlqvist} if the trees $+s$ and $-t$ are both $\epsilon$-regular Sahlqvist. An inequality $s \leq t$ is \emph{regular Sahlqvist} (DLR-Sahlqvist) if it is $\epsilon$-regular Sahlqvist for some $\epsilon$.
\end{definition}

It is easy to see that Sahlqvist DML-inequalities (cf.\ \cite{ConPal12, CoGhPa14}) are DLR-Sahlqvist, and moreover, if the DLR signature is specialized to unary maps which are normal, then DLR-Sahlqvist inequalities coincide with Sahlqvist DML-inequalities.
\begin{example}\label{eg:lemmonsahlqvist}

Lemmon's axioms (cf.\ Subsection \ref{subset:reg:modal:logic}) are examples of DLR-Sahlqvist formulas/inequalities. Indeed, the following axioms are DLR-Sahlqvist for the order-type $\epsilon_p=1$:
\[ (2)\ \Box p\to p, \quad\quad (4)\ \Box p\to\Box\Box p,\]
the following axiom is DLR-Sahlqvist for the order-type $\epsilon_p=\partial$:
 \[(5)\ \neg\Box p\to\Box\neg\Box p,\]
and the following axioms are DLR-Sahlqvist for the order-type $\epsilon_p=1,\epsilon_q=\partial$: \[(1)\ \Box(p\to q)\to\Box(\Box p\to\Box q)\mbox{ and }(1')\ \Box(p\to q)\to(\Box p\to\Box q).\]

Note that the axioms (1) and (1') are not DLR-Sahlqvist for the natural order-type $\epsilon_p=1,\epsilon_q=1$.
\end{example}
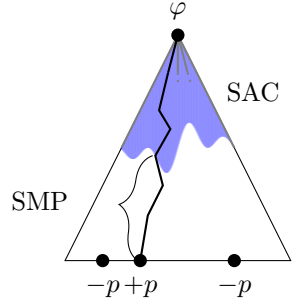
\begin{figure}
\begin{center}

\begin{tikzpicture}[node/.style={circle, draw, fill=black}, scale=0.5]
\draw (0,3) -- (3,-3) -- (-3, -3) -- cycle;
\foreach \x in {-1.5, -1.49,...,1.5}
 {
 \draw[blue, thick, draw opacity=0.1] (\x, {0.3*(sin(75.3*\x)+sin(75.3*\x*3)+sin(75.3*\x*5))}) -- (\x, {3 - abs(2*\x)});
 }

\draw[gray, thick] (0, 3) -- (-1.5, 0);
\draw[gray, thick] (0, 3) -- (1.5, 0);
\draw[gray, thick] (0, 3) -- (0, 2);
\draw[gray, thick, dotted] (0, 2) -- (0, 1.75);
\draw[black, thick] (0, 3) -- (-0.26, 2);
\draw[gray, thick] (0, 3) -- (0.26, 2);
\draw[gray, thick, dotted] (0.26, 2) -- (0.3349, 1.75);
\draw[black, thick] (-0.26, 2) -- (-0.5, 1) -- (-0.2, 0.5) -- (-0.6, {0.3*(sin(75.3*-0.6)+sin(75.3*-0.6*3)+sin(75.3*-0.6*5))}) -- (-0.4, -1) -- (-0.8, -1.8) -- (-1,-3);

\draw [decorate,decoration={brace,amplitude=10pt, mirror},xshift=-2pt,yshift=0pt]
(-0.6, {0.3*(sin(75.3*-0.6)+sin(75.3*-0.6*3)+sin(75.3*-0.6*5))}) -- (-1,-3) node [black,midway,xshift=-1.4cm, yshift=0.1cm]
{SMP};

\node[circle,fill=black!100,minimum size=5pt,inner sep=0pt, label=below:$+p$] (pos1) at (-1,-3) {};
\node[circle,fill=black!100,minimum size=5pt,inner sep=0pt, label=below:$-p$] (neg1) at (-2,-3) {};
\node[circle,fill=black!100,minimum size=5pt,inner sep=0pt, label=below:$-p$] (neg2) at (1.5, -3) {};
\node[circle,fill=black!100,minimum size=5pt,inner sep=0pt, label=above:$\varphi$] (root) at (0, 3) {};
\node[thick,standard/.style={draw=none, fill=none}] (Operators) at (2, 1.5) {{SAC}};

\end{tikzpicture}
\caption{Signed generation tree for regular Sahlqvist antecedent}

\end{center}
\end{figure}

\subsection{Canonicity for Sahlqvist DLR-inequalities}
\label{ssec:DLR canonicity}

\begin{theorem}\label{thm:additivecan}
Every regular Sahlqvist inequality is canonical.
\end{theorem}
\begin{proof}(Sketch)
The proof is similar to the proof of Theorem \ref{thm:canonicity:Sahlqvist}. All the steps in the proof of Theorem \ref{thm:canonicity:Sahlqvist} remain the same except \cite[Lemma 5.5 and 5.10]{GeNaVe05}, the analogues of which we prove below in Lemma \ref{lem:additive:easy} and Lemma \ref{lem:additive:difficult} respectively.
\end{proof}

\begin{lemma}\label{lem:additive:easy}
Every uniform DLR-term is both $\sigma$-contracting and $\pi$-expanding.
\end{lemma}
\begin{proof}
Let $\gamma$ be a uniform DLR-term. We show the lemma by induction on the complexity of $\gamma$. The cases involving logical connectives different from $k$ and $l$ are taken care of by \cite[Lemma 5.5]{GeNaVe05}. As to the remaining cases, let $\gamma=h(\gamma_1,\ldots,\gamma_n)$ such that $h$ is either $k$ or $l$. By the induction hypothesis, $\gamma_1$,$\ldots$, $\gamma_n$ are $\sigma$-contracting and $\pi$-expanding. Then the following chain of inequalities holds:
\begin{center}
\begin{tabular}{llr}
$(\gamma^{\A})^{\sigma}$&$=((h^\A(\gamma_1^\A,\ldots,\gamma_n^\A))^{\sigma}$ &\\

&$\leq(h^{\A})^\sigma(\gamma_1^\A,\ldots,\gamma_n^\A)^\sigma$ & (Theorem \ref{thm:ribiero})\\
&$=(h^{\A})^\sigma((\gamma_1^\A)^\sigma,\ldots,(\gamma_n^\A)^\sigma)$ & (Definition \ref{def:canonical:extension:maps})\\
&$\leq (h^{\A})^\sigma(\gamma_1^{\ca},\ldots,\gamma_n^{\ca})$& (induction hypothesis)\\

&$\leq h^{\ca}(\gamma_1^{\ca},\ldots,\gamma_n^{\ca})=\gamma^{\ca}$& (Definition \ref{def:canextdlr})\\
\end{tabular}
\end{center}
The last inequality reduces to an equality if $h = k$. When $h = l$, we also use the fact that $l^\sigma\leq l^\pi$ (see page \pageref{fsigmaleqfpi}).
The proof that $\gamma$ is $\pi$-expanding is similar.
\end{proof}
\begin{lemma}[\cite{HBMoL}, Proposition 117.5]\label{thm:mulsuc}

If $f:\A\rightarrow\B$ is multiplicative, then $f^\sigma$ is SUC.

\end{lemma}

\begin{cor}
\label{cor: mult compositional}
If $f:\A^{\delta}\to\B^{\delta}$ is multiplicative and $g:\B^{\delta}\to\C^{\delta}$ is order-preserving, then $g^{\sigma}f^{\sigma}\leq (gf)^{\sigma}$.

\end{cor}
The proof of the corollary above is very similar to that of Corollary \ref{cor:suc:GNV}, and uses Lemma \ref{thm:mulsuc} where Lemma \ref{lem:meet-pres_suc} was used. The details are omitted.

\begin{lemma}\label{lem:additive:difficult}
Every $\epsilon$-uniform DLR-Sahlqvist term is $\sigma$-expanding, and every $\epsilon^\partial$-uniform DLR-Sahlqvist term is $\pi$-contracting.
\end{lemma}
\begin{proof}Let $\alpha$ be an $\epsilon$-uniform DLR-Sahlqvist term. The proof is by induction on the complexity of $\alpha$. The cases when $\alpha$ is a propositional variable or a constant are trivial. Let $\alpha = k(\alpha_1,\alpha_2)$, with $k$ being a binary additive map of order-type $(\partial, 1)$. Then, by Definition \ref{def:rsahlqvist}, $\alpha_1$ is $\epsilon^\partial$-uniform DLR-Sahlqvist and $\alpha_2$ is $\epsilon$-uniform DLR-Sahlqvist. By induction hypothesis, $\alpha_1$ is $\pi$-contracting and $\alpha_2$ is $\sigma$-expanding. Moreover, by Definition \ref{def:canextdlr}, $k^{\ca}=(k^{\A})^{\sigma}$, i.e.,\ $k$ is $\sigma$-stable. Therefore, we have the following chain of inequalities:

\begin{center}

\begin{tabular}{r l l}
$(\alpha^{\A})^{\sigma}$ & $\geq (k^{\A})^{\sigma}(\alpha_1^{\A},\alpha_2^{\A})^{\sigma}$ & (Corollary \ref{lem:gen:Scott})\\
& $=k^{\ca}(\alpha_1^{\A},\alpha_2^{\A})^{\sigma}$& ($k$ is $\sigma$-stable)\\
& $=k^{\ca}((\alpha_1^{\A})^{\pi},(\alpha_2^{\A})^{\sigma})$ & (Definition \ref{def:canonical:extension:maps}, and $k$ is $(\partial, 1)$-additive)\\
& $\geq k^{\ca}(\alpha_1^{\ca},\alpha_2^{\ca})$ & (induction hypothesis)\\
& $=\alpha^{\ca}.$&\\
\end{tabular}
\end{center}

Let $\alpha = l(\alpha_1,\alpha_2)$ (notice that this case would only appear when $\alpha$ is $\epsilon^{\partial}$-uniform DLR-Sahlqvist), with $l$ being a binary multiplicative map of order-type $(\partial, 1)$. Then, by Definition \ref{def:rsahlqvist}, $\alpha_1$ is $\epsilon$-uniform DLR-Sahlqvist and $\alpha_2$ is $\epsilon^\partial$-uniform DLR-Sahlqvist. By induction hypothesis, $\alpha_1$ is $\sigma$-expanding and $\alpha_2$ is $\pi$-contracting. Moreover, by Definition \ref{def:canextdlr}, $l^{\ca}=(l^{\A})^{\pi}$, i.e.,\ $l$ is $\pi$-stable. Therefore, the following chain of inequal  ities holds:

\begin{center}

\begin{tabular}{r l l}
$(\alpha^{\A})^{\pi}$ & $\leq (l^{\A})^{\pi}(\alpha_1^{\A},\alpha_2^{\A})^{\pi}$ & (Corollary \ref{lem:gen:Scott})\\
& $=l^{\ca}(\alpha_1^{\A},\alpha_2^{\A})^{\pi}$& ($l$ is $\pi$-stable)\\
& $=l^{\ca}((\alpha_1^{\A})^{\sigma},(\alpha_2^{\A})^{\pi})$ & (Definition \ref{def:canonical:extension:maps}, and $l$ is $(\partial, 1)$-multiplicative)\\
& $\leq l^{\ca}(\alpha_1^{\ca},\alpha_2^{\ca})$ & (induction hypothesis)\\
& $=\alpha^{\ca}.$&\\
\end{tabular}
\end{center}

The remaining cases in which the outermost connective is a SAC node are easier and are left to the reader.
Let $\alpha=g(\overline{\beta})$, with $g$ being an SMP $n$-ary map of order type $\epsilon$. Then by Definition \ref{def:rsahlqvist}, $\beta_i$ is a composition of SMP maps if $\epsilon_i = 1$ or of dual SMP maps if $\epsilon_i = \partial$. Hence, the product map $\overline{\beta^\A}$ is multiplicative into $\A^\epsilon$. Therefore the following chain of (in)equalities holds:

\begin{center}
\begin{tabular}{r l l}
$(\alpha^{\A})^{\sigma}$ & $=(g^\A(\overline{\beta^\A}))^\sigma$& \\
& $\geq (g^\A)^\sigma(\overline{\beta^\A})^\sigma$ & (Corollary \ref{cor: mult compositional})\\
& $=g^{\ca}(\overline{\beta^\A})^\sigma$& (Definition \ref{def:canextdlr})\\
& $\geq g^{\ca}(\overline{{\beta}^{\ca}})$ & (induction hypothesis)\\
& $=\alpha^{\ca}.$&
\end{tabular}
\end{center}
\end{proof}

\section*{Part II: Algorithmic correspondence and canonicity}
\addcontentsline{toc}{section}{\underline{Part II: Algorithmic correspondence and canonicity}}
\section{\textsf{ALBA} on regular BDL and HA expansions}
 The theory developed in Part I is not enough to account for the strong completeness and elementarity of Lemmon's logics. Indeed, canonicity alone does not provide any clue as to which first-order correspondent (if any) Lemmon's axioms have.
The present section is aimed at filling this gap. We adapt the calculus for correspondence ALBA (cf.\ \cite{ConPal12}) to the setting of regular modal logic, and we consider both the case in which the propositional base is given by the logic of bounded distributive lattices (BDLs) and by intuitionistic logic, the algebras associated with which are Heyting algebras (HAs). We also illustrate how the adapted calculus---which from now on we refer to as \textsf{ALBA}$^r$---works by means of an example.

The calculus \textsf{ALBA}$^r$ is very closely related to the metacalculus \textsf{ALBA}$^e$ defined in \cite[Section 5.1]{CGPSZ14}. Specifically, in that setting, the original \textsf{ALBA} was augmented with special approximations and adjunction rules for given unary term functions $\pi, \sigma, \lambda, \rho$ such that $\pi: \A\to \A$ and $\lambda: \A^{\partial}\to \A$ are additive and $\sigma: \A\to \A$ and $\rho: \A^{\partial}\to \A$ are multiplicative. The present calculus \textsf{ALBA}$^r$ for the DLR signature coincides with the restriction of \textsf{ALBA}$^e$ to the special case in which each term function $\pi, \sigma, \lambda, \rho$ reduces to one (fresh) connective. This restriction behaves significantly better than \textsf{ALBA}$^e$. Indeed, in \cite[Proposition 29]{CGPSZ14}, the soundness of the special rules was shown with respect to perfect DLEs which are canonical extensions of some DLEs. It was also discussed (cf.\ \cite[Section 9]{CGPSZ14}) that this soundness result cannot be strengthened to arbitrary perfect DLEs, and that this implies that the relativized canonicity result shown via \textsf{ALBA}$^e$ cannot be improved to a relativized correspondence result.
As we will discuss below, all rules of \textsf{ALBA}$^r$ are sound with respect to all perfect DLRs. That is, the limitation preventing \textsf{ALBA}$^e$ from yielding correspondence results does not hold for \textsf{ALBA}$^r$.

The proof of soundness of \textsf{ALBA}$^r$ and the canonicity of each inequality on which \textsf{ALBA}$^r$ succeeds are very similar to the proofs of soundness and relativized canonicity proved in \cite{CGPSZ14}, so we will only expand on the strengthening mentioned above. The HAR setting is very similar and will be discussed as we go along. In Part III, the algorithm \textsf{ALBA}$^r$ will be successfully run on Lemmon's axioms, and their strong completeness and elementarity will follow from the results of the presnt section.

In what follows, we define \textsf{ALBA}$^r$ for DLR$_{\eta}$ and HAR$_{\eta}$ simultaneously and from first principles. We use the symbol $\mathcal{L}$ to refer indifferently to DLR or to HAR.

\subsection{The expanded language $\mathcal{L}^{+}$}
\label{ssec:expanded language}
Analogously to what has been done in \cite{ConPal12}, we need to introduce the expanded language $\mathcal{L}^+_\eta$ the calculus \textsf{ALBA}$^r$ will manipulate. The language $\mathcal{L}^+_\eta$ will be shaped on the perfect $\mathcal{L}$-algebras appropriate to each $\eta$. Indeed, as usual, $\mathcal{L}^+_\eta$ will be built on three pairwise disjoint sets of variables: proposition variables in $\mathsf{AtProp}$ (denoted by $p, q, r$), {\em nominal} variables in $\mathsf{Nom}$ (denoted by $\nomi, \nomj, \nomk$) and {\em conominal} variables in $\mathsf{CNom}$ (denoted by $\cnomm, \cnomn, \cnomo$). Nominals and conominals are to be interpreted as completely join-irreducible and completely meet-irreducible elements of perfect DLRs.
As discussed in \cite{CGPSZ14}, every additive map $f: \A\to \B$ between perfect BDLs can be associated with its {\em normalization}, that is a map\footnote{When $\A = (\A')^\partial$, we will rather use the symbols ${\lhd}_f$ and ${\rhd}_g$ to respectively denote the normalizations of $f$ and $g$ as maps $\A'\to \B$.} $\Diamond_f: \A\to \B$ such that, for every $u\in \A$,
\[\Diamond_f u = \bigvee\{j\in J^{\infty}(\B)\mid j\leq f(i) \mbox{ for some } i\in J^{\infty}(\A) \mbox{ such that } i\leq u\}.\]

Order-dually, every multiplicative map $g: \A\to \B$ between perfect BDLs can be associated with its {\em normalization}, that is a map $\Box_g: \A\to \B$ such that, for every $u\in \A$,
\[\Box_g u = \bigwedge\{n\in M^{\infty}(\B)\mid g(m)\leq n \mbox{ for some } m\in M^{\infty}(\A) \mbox{ such that } u\leq m\}.\]

By definition, the normalizations of $f$ and $g$ are completely join-preserving and completely meet-preserving respectively. Since perfect lattices are complete, this implies that the normalizations are adjoints, i.e., there exist maps\footnote{Using the alternative notation, there exist maps ${\blacktriangleleft}_f, {\blacktriangleright}_g: \B\to \A$ such that for every $u\in \A$ and $v\in \B$,

\[{\lhd}_f u\leq v \ \mbox{ iff }\ {\blacktriangleleft}_f v\leq u \quad \quad \quad v\leq {\triangleright}_g u \ \mbox{ iff }\ u\leq {\blacktriangleright}_g v \] } $\blacksquare_f, \Diamondblack_g: \B\to \A$ such that for every $u\in \A$ and $v\in \B$,

\[\Diamond_f u\leq v \ \mbox{ iff }\ u\leq \blacksquare_f v \quad \quad \quad \Diamondblack_g v\leq u \ \mbox{ iff }\ v\leq \Box_g u.\]

The discussion above motivates the following recursive definition of the formulas in the expanded language $\mathcal{L}^+_\eta$ (which is the same for both DLR and HAR):
\begin{equation}\label{def:language}
\varphi ::= \; \bot \mid \top \mid p \mid \nomj \mid \cnomm \mid
\varphi \vee \varphi \mid \varphi \wedge \varphi \mid
\varphi - \varphi \mid
\varphi \rightarrow \varphi \mid
f(\varphi) \mid g(\varphi) \mid
\end{equation}
\begin{equation}
k(\overline{\varphi}) \mid l(\overline{\varphi}) \mid\begin{cases} \blacksquare_f \varphi & \mbox{ if }\eta_f = 1\\
{\blacktriangleleft}_f\varphi & \mbox{ if } \eta_f = \partial \end{cases}
 \mid \begin{cases} \Diamondblack_g \varphi & \mbox{ if }\eta_g = 1\\
{\blacktriangleright}_g\varphi & \mbox{ if } \eta_g = \partial \end{cases}
\end{equation}
where $p\in \mathsf{AtProp}$, $\nomj \in \mathsf{Nom}$, $\cnomm \in \mathsf{CNom}$ and $\eta\in\{1,\partial\}^{2}$ is a 2-order-type.

The inclusion of the additional propositional connectives $-$ and $\rightarrow$ in the signature above is motivated by the well known fact that perfect DLRs have a natural structure of bi-Heyting algebras.

Recall that a \emph{quasi-inequality} is an expressions of the form $(\varphi_1\leq \psi_1 \&\ldots \& \varphi_n\leq\psi_n) \Rightarrow \varphi\leq \psi$, where the $\varphi$'s and $\psi$'s are terms from $\mathcal{L}^{+}_\eta$. Formulas, inequalities and quasi-inequalities which do not contain any propositional variables are called \emph{pure}. In what follows, we will also use the connective $\parr$ to denote (meta-)disjunction in the context of quasi-inequalities.

The interpretation of the modal operators is the natural one suggested by the notation, and indeed we are using the same symbols to denote both the logical connectives and their algebraic interpretations.

\subsection{The algorithm \textsf{ALBA}$^r$}

In what follows, we illustrate how \textsf{ALBA}$^r$ works, while at the same time we introduce its rules. The proof of the soundness and invertibility of the general rules for the BDL setting is discussed in \cite{ConPal12, CoGhPa14}. The proof of the soundness and invertibility of the rules which are specific to the regular setting follows from the soundness and invertibility of the rules for $\pi, \sigma, \lambda, \rho$ of the metacalculus given in \cite{CGPSZ14}, which is there discussed in detail. We refer the reader to these discussions, and we do not elaborate further on this topic.

\textsf{ALBA}$^r$ manipulates input $\mathcal{L}$-inequalities $\phi\leq \psi$ and proceeds in three stages:

\paragraph{First stage: preprocessing and first approximation.}

\textsf{ALBA}$^r$ preprocesses the input inequality $\phi\leq \psi$ by performing the following steps
exhaustively in the signed generation trees $+\phi$ and $-\psi$:

\begin{enumerate}
\item
\begin{enumerate}
\item Push down, towards variables, occurrences of $+\land$, $+f$ for $\eta_f = 1$, $-g$ for $\eta_g = \partial$, $+k$ for $\epsilon_k(i)=1$, $-l$ for $\epsilon_l(i)=\partial$ by distributing them over nodes labelled with $+\lor$ (in the $i$-th coordinate for $k, l$) which are SAC nodes, and

\item Push down, towards variables, occurrences of $-\lor$, $-g$ for $\eta_g =1$,  $+ f$ for $\eta_f = \partial$, $+k$ for $\epsilon_k(i)=\partial$, $-l$ for $\epsilon_l(i)=1$ by distributing them over nodes labelled with $-\land$ (in the $i$-th coordinate for $k, l$) which are SAC nodes.

\end{enumerate}

\item Apply the splitting rules:

$$\infer{\alpha\leq\beta\ \ \ \alpha\leq\gamma}{\alpha\leq\beta\land\gamma}
\qquad
\infer{\alpha\leq\gamma\ \ \ \beta\leq\gamma}{\alpha\lor\beta\leq\gamma}
$$

\item Apply the monotone and antitone variable-elimination rules:

$$\infer{\alpha(\perp)\leq\beta(\perp)}{\alpha(p)\leq\beta(p)}
\qquad
\infer{\beta(\top)\leq\alpha(\top)}{\beta(p)\leq\alpha(p)}
$$

for $\beta(p)$ positive in $p$ and $\alpha(p)$ negative in $p$.

\end{enumerate}

Let $\mathsf{Preprocess}(\phi\leq\psi)$ be the finite set $\{\phi_i\leq\psi_i\mid 1\leq i\leq n\}$ of inequalities obtained after the exhaustive application of the previous rules. We proceed separately on each of them, and hence, in what follows, we focus only on one element $\phi_i\leq\psi_i$ in $\mathsf{Preprocess}(\phi\leq\psi)$, and we drop the subscript. Next, the following {\em first approximation rule} is applied {\em only once} to every inequality in $\mathsf{Preprocess}(\phi\leq\psi)$:

$$\infer{\nomi_0\leq\phi\ \ \ \psi\leq \cnomm_0}{\phi\leq\psi}
$$

Here, $\nomi_0$ and $\cnomm_0$ are a nominal and a conominal respectively. The first-approximation
step gives rise to systems of inequalities $\{\nomi_0\leq\phi_i, \psi_i\leq \cnomm_0\}$ for each inequality in $\mathsf{Preprocess}(\phi\leq\psi)$. Each such system is called an {\em initial system}, and is now passed on to the reduction-elimination cycle.

\paragraph{Second stage: reduction-elimination cycle.}

The goal of the reduction-elimination cycle is to eliminate all propositional variables from the systems
which it receives from the preprocessing phase. The elimination of each variable is effected by an
application of one of the Ackermann rules given below. In order to apply an Ackermann rule, the
system must have a specific shape. The adjunction, residuation, approximation, and splitting rules are used to transform systems into this shape. The rules of the reduction-elimination cycle, viz.\ the adjunction, residuation, approximation, splitting, and Ackermann rules, will be collectively called the {\em reduction} rules.

\paragraph{Residuation rules. }
\begin{center}
\begin{tabular}{cc}
\AxiomC{$\phi\wedge \psi\leq\chi$}
\doubleLine
\UnaryInfC{$\psi\leq \phi\rightarrow \chi$}
\DisplayProof

&
\AxiomC{$ \phi\leq \psi\vee \chi$}

\UnaryInfC{$\phi - \chi\leq\psi$}
\DisplayProof

\end{tabular}
\end{center}

In the HAR setting, the rule on the left-hand side above is allowed to be executed bottom-to-top.

\paragraph{Adjunction rules.}

\begin{center}
\begin{tabular}{cc}
\AxiomC{$f(\phi)\leq\psi$}
\RightLabel{(if $\eta_f = 1$)}
\UnaryInfC{$f(\bot)\leq \psi\;\;\;\phi\leq\blacksquare_{f}\psi$}
\DisplayProof

&
\AxiomC{$ \phi\leq g(\psi)$}
\RightLabel{(if $\eta_g = 1$)}
\UnaryInfC{$\phi\leq g(\top)\;\;\; \Diamondblack_g\phi\leq\psi$}
\DisplayProof

\end{tabular}
\end{center}

\begin{center}
\begin{tabular}{cc}
\AxiomC{$f(\phi)\leq \psi$}
\RightLabel{(if $\eta_f = \partial$)}
\UnaryInfC{$f(\top)\leq\psi\;\;\; {\blhd} \psi\leq\phi$}
\DisplayProof

& \AxiomC{$\phi\leq g(\psi)$}
\RightLabel{(if $\eta_g = \partial$)}
\UnaryInfC{$\phi\leq g(\bot)\;\;\;\psi\leq {\brhd} \phi$}
\DisplayProof \\

\end{tabular}
\end{center}
In a given system, each of these rules replaces an instance of the upper inequality with the corresponding instances of the two lower inequalities.

The leftmost inequalities in each rule above will be referred to as the \emph{side condition}\label{def:sidecondition}.

\paragraph{Approximation rules.} The following rules are applicable if $\eta_f=\eta_g= 1$:
\begin{center}
\begin{tabular}{cccc}

\AxiomC{$\nomi\leq f(\phi)$}
\UnaryInfC{$[\nomi\leq f(\bot)]\;\;\parr\;\; [\nomj\leq\phi\;\;\;\;\nomi\leq f(\nomj)]$}
\DisplayProof

&
\AxiomC{$ g(\psi)\leq\cnomm$}
\UnaryInfC{$[g(\top)\leq \cnomm]\;\;\parr\;\; [\psi\leq\cnomn\;\;\;\;g(\cnomn)\leq\cnomm]$}
\DisplayProof\\\\

\end{tabular}
\end{center}

If $\eta_f=\eta_g=\partial$, the following approximation rules are applicable:

\begin{center}
\begin{tabular}{cccc}

\AxiomC{$\nomi\leq f(\phi)$}
\UnaryInfC{$[\nomi\leq f(\top)]\;\;\parr\;\; [\phi\leq\cnomm\;\;\;\; \nomi\leq f(\cnomm)]$}
\DisplayProof

& \AxiomC{$g(\psi)\leq\cnomm$}
\UnaryInfC{$[g(\bot)\leq\cnomm]\;\;\parr\;\; [\nomi\leq\psi\;\;\;\;g(\nomi)\leq\cnomm]$}
\DisplayProof \\\\

\end{tabular}
\end{center}
\begin{center}
\begin{tabular}{l}

\AxiomC{$\nomi\leq k(\overline{\phi}_{\epsilon_{k}^+}, \overline{\psi}_{\epsilon_{k}^-})$}

\UnaryInfC{$\bigparr\limits_{P\subseteq\epsilon_{k}^+, N\subseteq\epsilon_{k}^-}(\nomi\leq k(\overline{\nomj}_{P}, \overline{\bot}_{\epsilon_{k}^+\setminus P}, \overline{\cnomm}_{N}, \overline{\top}_{\epsilon_{k}^-\setminus N})\;\;\&\;\;\bigamp\limits_{e\in P}(\nomj_e\leq\phi_e)\;\;\&\;\;\bigamp\limits_{e\in N}(\psi_e\leq\cnomm_e))$}
\DisplayProof\\\\

\AxiomC{$l(\overline{\phi}_{\epsilon_{l}^+}, \overline{\psi}_{\epsilon_{l}^-})\leq\cnomm$}

\UnaryInfC{$\bigparr\limits_{P\subseteq\epsilon_{l}^+, N\subseteq\epsilon_{l}^-}(l(\overline{\cnomn}_{P}, \overline{\top}_{\epsilon_{l}^+\setminus P}, \overline{\nomi}_{N}, \overline{\bot}_{\epsilon_{l}^-\setminus N})\leq \cnomm\;\;\&\;\;\bigamp\limits_{e\in P}(\phi_e\leq\cnomn_e)\;\;\&\;\;\bigamp\limits_{e\in N}(\nomi_e\leq\psi_e))$}
\DisplayProof\\\\

\end{tabular}
\end{center}

The leftmost inequalities in each rule above will be referred to as the \emph{side condition}.
In the rules above, we adopt the following notation: by $k(\overline{\phi}_{\epsilon_{k}^+}, \overline{\psi}_{\epsilon_{k}^-})$, we have grouped together the positive and the negative coordinates of the connective $k$.

Each approximation rule transforms a given system $S\cup\{s\leq t\}$ into systems $S\cup\{s_1\leq t_1\}$ and $S\cup\{s_2\leq t_2, s_3\leq t_3\}$, the first of which containing only the side condition (in which no propositional variable occurs), and the second one containing the instances of the two remaining lower inequalities.

The nominals and conominals introduced by the approximation rules must be {\em fresh}, i.e.\ must not already occur in the system before applying the rule.

\paragraph{Additional approximation rules for the HAR setting.}
\begin{center}
\begin{tabular}{cc}
\AxiomC{$\phi\rightarrow \chi\leq \cnomm $}
\RightLabel{($\rightarrow$Appr$_1$)}
\UnaryInfC{ $\nomj\leq \phi\ \&\ \chi\leq \cnomn\ \&\ \nomj\rightarrow \cnomn\leq \cnomm$}
\DisplayProof
\end{tabular}
\end{center}

\paragraph{Ackermann rules.} These rules are the core of \textsf{ALBA}$^r$, since their application eliminates proposition variables. As mentioned earlier, all the preceding steps are aimed at equivalently rewriting the input system into one or more systems, each of which of a shape in which the Ackermann rules can be applied. An important feature of Ackermann rules is that they are executed on the whole set of inequalities in which a given variable occurs, and not on a single inequality.\\

$$\frac{ \exists p \; \Big [
\mbox{\Large{\&}}_{i=1}^{n} \{ \alpha_i \leq p \} \;\; \& \;\;
 \mbox{\Large{\&}}_{j=1}^m \{ \beta_j(p)\leq \gamma_j(p)\} \Big ]
}{
\mbox{\Large{\&}}_{j=1}^m
\{\beta_j(\bigvee_{i=1}^n \alpha_i)\leq \gamma_j(\bigvee_{i=1}^n \alpha_i) \} } (RAR)$$
where
$p$ does not occur in $\alpha_1, \ldots, \alpha_n$, the formulas
$\beta_{1}(p), \ldots, \beta_{m}(p)$ are positive in $p$, and
$\gamma_{1}(p), \ldots, \gamma_{m}(p)$ are negative in $p$.
Here below is the left-Ackermann rule:

$$\frac{\exists p \; \Big [
\mbox{\Large{\&}}_{i=1}^{n} \{ p \leq \alpha_i \} \;\; \& \;\;
 \mbox{\Large{\&}}_{j=1}^m \{ \beta_j(p)\leq \gamma_j(p)\} \Big ]
 }{
 \mbox{\Large{\&}}_{j=1}^m \{ \beta_j(\bigwedge_{i=1}^n \alpha_i)\leq \gamma_j(\bigwedge_{i=1}^n \alpha_i) \}}
(LAR)$$
where
$p$ does not occur in $\alpha_1, \ldots, \alpha_n$, the formulas
$\beta_{1}(p), \ldots, \beta_{m}(p)$ are negative in $p$, and
$\gamma_{1}(p), \ldots, \gamma_{m}(p)$ are positive in $p$.

\paragraph{Third stage: output.}

If there was some system in the second stage from which not all occurring propositional variables could be eliminated through the application of the reduction rules, then \textsf{ALBA}$^r$ reports failure and terminates. Else, each system $\{\nomi_0\leq\phi_i, \psi_i\leq \cnomm_0\}$ obtained from $\mathsf{Preprocess}(\varphi\leq \psi)$ has been reduced to a system, denoted $\mathsf{Reduce}(\varphi_i\leq \psi_i)$, containing no propositional variables. Let \textsf{ALBA}$^r$$(\varphi\leq \psi)$ be the set of quasi-inequalities \begin{center}{\Large{\&}}$[\mathsf{Reduce}(\varphi_i\leq \psi_i) ]\Rightarrow \nomi_0 \leq \cnomm_0$\end{center} for each $\varphi_i \leq \psi_i \in \mathsf{Preprocess}(\varphi\leq \psi)$.

Notice that all members of \textsf{ALBA}$^r(\varphi\leq \psi)$ are free of propositional variables. Hence, translating them as discussed in Section \ref{sec:example:lemmon:corre:canoni} produces sentences in the language of the Kripke frames with impossible worlds dual to the perfect r-BAEs which hold simultaneously on those Kripke frames with impossible worlds iff the input inequality is valid on them.
\textsf{ALBA}$^r$ returns \textsf{ALBA}$^r(\varphi\leq \psi)$ and terminates. An inequality $\varphi\leq \psi$ on which \textsf{ALBA}$^r$ succeeds will be called an \textsf{ALBA}$^r$-{\em inequality}.
\subsection{Soundness and canonicity of \textsf{ALBA}$^r$}\label{subsec:correctness}
As we mentioned at the beginning of the present section, \textsf{ALBA}$^r$ coincides with the restriction of the metacalculus defined in \cite[Section 5.1]{CGPSZ14} to the special case in which each term function $\pi, \sigma, \lambda, \rho$ reduces to one (fresh) connective. Hence, the proof of canonicity is an immediate consequence of the present treatment relies on \cite[Section 5.1]{CGPSZ14}, and only expands on the points in which \textsf{ALBA}$^r$ and \textsf{ALBA}$^e$ behave differently.

\begin{theorem}[Soundness]\label{Crctns:Theorem}

For any $\mathsf{ALBA}^r$-inequality $\varphi\leq \psi$, for any perfect $\mathcal{L}$-algebra $\A$, \[\A\models\varphi\leq\psi \mbox{ iff } \A\models\mathsf{ALBA}^r(\varphi\leq \psi). \]

\end{theorem}
\begin{proof}
This proof is very similar to the proof of \cite[Proposition 29]{CGPSZ14}. The only rules that deserve additional discussion are the approximation and the adjunction rules, the soundness of which can be proved, as in \cite[Proposition 29]{CGPSZ14}, by deriving them from a set of rules which includes the following ones, which, for the sake of conciseness, are given as formula-rewriting rules:

\begin{center}
\begin{tabular}{cc}

\AxiomC{$f(p)$}
\LeftLabel{($\eta_f = 1$)}
\UnaryInfC{$f(\bot)\vee \diam_f(p)$}
\DisplayProof

&
\AxiomC{$ g(p)$}
\RightLabel{($\eta_g = 1$)}
\UnaryInfC{$g(\top)\wedge \Box_g(p)$}
\DisplayProof\\

& \\
\AxiomC{$f(p)$}
\LeftLabel{($\eta_f = \partial$)}
\UnaryInfC{$f(\top)\vee {\lhd}_f(p)$}
\DisplayProof
& \AxiomC{$g(p)$}
\RightLabel{($\eta_g = \partial$)}
\UnaryInfC{$g(\bot)\wedge {\rhd}_g (p)$}
\DisplayProof \\
\end{tabular}\end{center}
Notice that in any perfect DLR $\A$, the interpretations of the connectives $f$ and $g$ are completely additive and multiplicative, respectively. Hence, it is easy to see that

\[\A \models f(p) = f(\bot) \lor \Diamond_f(p) \mbox{ if } \eta_f = 1 \quad\quad \A \models g(p) = g(\top) \land \Box_g(p) \mbox{ if } \eta_g = 1 \]
\[\A \models f(p) = f(\top) \lor {\lhd}_f(p) \mbox{ if } \eta_f = \partial \quad\quad \A \models g(p) = g(\bot) \land {\rhd}_g(p) \mbox{ if } \eta_g = \partial \]
which proves the soundness and invertibility of the rules above on any perfect DLR $\A$.
Let us give two derivations as examples: $\eta_f = \partial$ in the left-hand one and $\eta_f = 1$ in the right-hand one.

\begin{center}
\begin{tabular}{cc}
\AxiomC{$f(\phi)\leq \psi$}
\UnaryInfC{$f(\top)\vee{\lhd}_f(\phi)\leq \psi$}
\UnaryInfC{$f(\top)\leq \psi\;\;\; {\lhd}_f(\phi)\leq \psi$}
\UnaryInfC{$f(\top)\leq\psi\;\;\; {\blhd} \psi\leq\phi$}
\DisplayProof
&
\AxiomC{$\nomi\leq f(\psi)$}
\UnaryInfC{$\nomi\leq f(\bot)\vee \diam_f(\psi)$}
\UnaryInfC{$[\nomi\leq f(\bot)]\;\;\parr\;\; [\nomi\leq\diam_f(\psi)]$}
\UnaryInfC{$[\nomi\leq f(\bot)]\;\;\parr\;\; [\nomj\leq\psi\;\;\;\;\nomi\leq\Diamond_f(\nomj)]$}
\UnaryInfC{$[\nomi\leq f(\bot)]\;\;\parr\;\; [\nomj\leq\psi\;\;\;\;\nomi\leq f(\nomj)]$}

\DisplayProof
\\
\end{tabular}
\end{center}
\end{proof}

\begin{dfn}
\label{def:safe:execution}
An execution of $\mathsf{ALBA}^r$ is \emph{safe} if no side conditions (cf.\ Page \pageref{def:sidecondition}) introduced by applications of adjunction rules for the new connectives are further modified, except for receiving Ackermann substitutions.
\end{dfn}

\begin{theorem}\label{thm:canonical}
All inequalities on which $\mathsf{ALBA}^r$ safely succeeds are canonical.
\end{theorem}
\begin{proof}
Straightforward adaptation of the proof of \cite[Theorem 38]{CGPSZ14}.
\end{proof}
Before moving on, let us briefly discuss the specific HAR-setting with respect to canonicity. The only difference between HAR and DLR is that the logical connective $\rightarrow$ is part of the original signature. Hence, in the HAR setting, the connective $\rightarrow$ has better topological/order-theoretic properties than it has in the DLR setting, the difference being comparable to that between the `white' and the `black' connectives in \cite{ConPal12}.

In particular, one needs to show that the additional rules for $\rightarrow$ are sound and invertible under both arbitrary and admissible valuations, and that they preserve and reflect compact appropriateness (cf.\ \cite[Definition 41]{CGPSZ14}), and when applied safely, they preserve topological adequacy (cf.\ \cite[Definition 40]{CGPSZ14}).\footnote{\label{footnote top adeq}
A system $S$ of $\mathcal{L}_{\eta}^{+}$ inequalities is {\em topologically adequate} if whenever $\blacksquare_f\psi$ (resp.\ $\Diamondblack_g\phi$, $\blacktriangleright_f\psi$, $\blacktriangleleft_g\phi$) occurs in $S$, then $f(\bot)\leq \psi$ (resp.\ $g(\top)\geq \phi$, $f(\bot)\geq \psi$, $g(\top)\leq \phi$) is in $S$.
A system $S$ of $\mathcal{L}_{\eta}^{+}$ inequalities is {\em compact-appropriate} if the left-hand side of each inequality in $S$ is syntactically closed and the right-hand side of each inequality in $S$ is syntactically open (cf.\ \cite[Definition 46]{CGPSZ14}).} The first and second requirements are of straightforward verification. The third requirement immediately follows from the definitions of safe execution and topological adequacy (cf.\ footnote \ref{footnote top adeq}).

\section{Inductive inequalities, and success of \textsf{ALBA}$^r$ on them}
In the present section, we define the class of inductive DLR- and HAR-inequalities (from now on abbreviated as inductive inequalities, unless we need to distinguish the two languages), and prove that \textsf{ALBA}$^r$ succeeds on each of them with a safe run. In the light of Theorems \ref{Crctns:Theorem} and \ref{thm:canonical}, this will show that inductive inequalities are elementary and canonical, which in particular strengthens the J\'onsson-style canonicity of Part I, given that inductive inequalities have Sahlqvist inequalities as a proper subclass. Unlike the corresponding definitions in \cite[Section 3]{ConPal12}, the definitions below are given in terms of the positive classification (cf.\ \cite[Section 6.2]{CoGhPa14}).

\subsection{Inductive DLR- and HAR-inequalities}

\begin{table}[\here]
\begin{center}
\begin{tabular}{| c | c|}
\hline
Skeleton & PIA\\
\hline
SAC & SMP \\

\begin{tabular}{ c c c c c c}
$+$ &&$\vee$ &$\wedge$ &$f$ & $k$ \\
$-$ &$\rightarrow$ &$\wedge$ & $\vee$ & $g$ & $l$\\
\end{tabular}
&

\begin{tabular}{ c c c}
$+$ &$\wedge$&$g$\\
$-$ &$\vee$ & $f$\\

\end{tabular}\\
\hline
& SRR \\
&
\begin{tabular}{ c c c}
$+$ & $\vee$ &$\rightarrow$\\
$-$ &$\wedge$ & \\

\end{tabular}\\
\hline
\end{tabular}
\end{center}
\caption{Classification of nodes for $\mathrm{HAR}$ and $\mathrm{DLR}$. }
\label{Join:and:Meet:Friendly:Table:HAR:DLR}
\end{table}

The following definitions make use of the auxiliary Definition \ref{Excellent:Branch:Def}. This definition is to be relativized to the present DLR/HAR-settings, by making use of Table \ref{Join:and:Meet:Friendly:Table:HAR:DLR}.

\begin{definition}[Inductive DLR-inequalities]
\label{Inductive DLR:Ineq:Def}
For any order type $\epsilon$ and any irreflexive and transitive relation $\Omega$ on $p_1,\ldots p_n$, the (negative or positive) generation tree $*s$ $(* \in \{-, + \})$ of a term $s(p_1,\ldots p_n)$ is \emph{$(\Omega, \epsilon)$-inductive}
if every $\epsilon$-critical branch with leaf labelled $p_i$ is good (cf.\ Definition \ref{Excellent:Branch:Def}), and moreover, for every binary SRR node $\ast (\alpha \circ \beta)$,
\begin{enumerate}
\item $\epsilon^\partial(\star\alpha)$, and
\item $p_j <_{\Omega} p_i$ for every $p_j$ occurring in $\alpha$.
\end{enumerate}
Clearly, the conditions above imply that the $\epsilon$-critical branch runs through $\beta$.
We will refer to $\Omega$ as the \emph{dependency order} on the variables. An inequality $s \leq t$ is \emph{$(\Omega, \epsilon)$-inductive} if the trees $+s$ and $-t$ are both $(\Omega, \epsilon)$-inductive. An inequality $s \leq t$ is \emph{inductive} if it is $(\Omega, \epsilon)$-inductive for some $\Omega$ and $\epsilon$.
\end{definition}

\begin{definition}[Inductive HAR-inequalities]
\label{Inductive HAR:Ineq:Def}
Given an order type $\epsilon$, and an irreflexive and transitive relation $\Omega$ on the variables $p_1,\ldots p_n$, the (negative or positive) generation tree $*s$, $* \in \{-, + \}$, of a term $s(p_1,\ldots p_n)$ is \emph{$(\Omega, \epsilon)$-inductive}
if every $\epsilon$-critical branch with leaf labelled $p_i$ is good (cf.\ Definition \ref{Excellent:Branch:Def}), and moreover, for every binary SRR node $\ast (\alpha \circ \beta)$,

\begin{enumerate}
\item $\epsilon^\partial(\star\alpha)$ (resp.\ $\epsilon(\star\alpha)$) if $\circ$ is positive (resp.\ negative) in the first coordinate, and
\item $p_j <_{\Omega} p_i$ for every $p_j$ occurring in $\alpha$.
\end{enumerate}

Clearly, the conditions above imply that $\epsilon$-critical branches can pass only through the positive coordinate of $\rightarrow$.
The definition of $(\Omega, \epsilon)$-inductive and inductive inequalities is analogous to the previous definition.
\end{definition}

Notice that, in the DLR-signature, SRR nodes can be either $+\vee$ or $-\wedge$, whereas in the HAR-signature, SRR nodes can also be $+\rightarrow$, and hence the corresponding subtree of $*s$ is either $+ (\alpha \vee \beta)$ or $-(\alpha \wedge \beta)$ or $+(\alpha\rightarrow \beta)$. In each of the two settings, since excellent branches are in particular good, it is easy to see that Sahlqvist inequalities are special inductive inequalities.

\begin{example}
Regarded as HAR-formulas/inequalities, the following axioms from Section \ref{sec:prelimi},
\[(1)\ \Box(p\to q)\to\Box(\Box p\to\Box q)\ \ \ \mbox{ and }\ \ \ (1')\ \Box(p\to q)\to(\Box p\to\Box q)\]
 which are Sahlqvist for the order-type $\epsilon_p=1, \epsilon_q=\partial$, are inductive, but not Sahlqvist, for the natural order-type $\epsilon_p=\epsilon_q=1$ and dependency order $p<_{\Omega}q$.

\end{example}

\subsection{\textsf{ALBA}$^r$ succeeds on inductive inequalities}
\label{sec:ALBA on inductive}
In the present subsection, we discuss the success of \textsf{ALBA}$^r$ in both the DLR and the HAR setting simultaneously. We will use the symbol $\mathcal{L}$ to refer generically to either of DLR or HAR. The treatment of the present subsection is very similar to that of \cite[Section 10]{ConPal12}, hence, in what follows, we expand only on details which are specific to the regular setting. Let us start with some auxiliary definitions and lemmas. Unlike the corresponding definitions in \cite[Section 10]{ConPal12}, the definitions below are given in terms of the positive classification (cf.\ \cite[Section 6.2]{CoGhPa14}).

\begin{definition}
\label{def:epsilon conservative}
Given an order type $\epsilon \in \{1, \partial\}^{n}$, a signed generation tree $\ast \phi$ of a term $\phi(p_1, \ldots, p_n) \in \mathcal{L}^+$ is called \emph{$\epsilon$-conservative} if all connectives occurring on $\epsilon$-critical branches of $\ast \phi$ are from the base language $\mathcal{L}$. 
\end{definition}
The next definition extends the notion of inductive terms and inequalities to $\mathcal{L}_{\mathit{term}}^+$ and $\mathcal{L}^+$, essentially by keeping Definition \ref{Inductive DLR:Ineq:Def} intact and simply forbidding connectives belonging properly to the extended language on critical branches. Nevertheless, since this definition will be used extensively, we write it out in full here.
\begin{definition}[$(\Omega, \epsilon)$-inductive $\mathcal{L}^+$-terms and inequalities]\label{Inductive+:Ineq:Def}
For any order type $\epsilon$ and any irreflexive and transitive relation $\Omega$ on $p_1,\ldots p_n$, the generation tree $*s$ $(* \in \{-, + \})$ of a $\mathcal{L}^+$-term $s(p_1,\ldots p_n)$ is \emph{$(\Omega, \epsilon)$-inductive} if
\begin{enumerate}
\item it is $\epsilon$-conservative, and

\item for each $1 \leq i \leq n$, on every $\epsilon$-critical branch with leaf labelled $p_i$ is good (cf.\ Definition \ref{Excellent:Branch:Def}), and moreover, every binary node $\star(\alpha\circ \beta)$ in $P_1$ for $\star\in \{+, -\}$, such that the critical branch lies in $\beta$ satisfies the following conditions:
\begin{enumerate}
\item $\epsilon^\partial(\star\alpha)$ (resp.\ $\epsilon(\star\alpha)$) if $\circ$ is positive (resp.\ negative) in the first coordinate, and
\item $p_j <_{\Omega} p_i$ for every $p_j$ occurring in $\alpha$.
\end{enumerate}
\end{enumerate}
An $\mathcal{L}^+$-inequality $s \leq t$ is \emph{$(\Omega, \epsilon)$-inductive} if the trees $+s$ and $-t$ are both $(\Omega, \epsilon)$-inductive.
\end{definition}
Notice that in the DLR setting, $\circ$ in item (a) above can only be $+\lor$ and $-\land$, and hence the options in brackets are ruled out.

The next definition makes use of auxiliary notions in Definition \ref{Excellent:Branch:Def}.

\begin{definition}[Definite $(\Omega, \epsilon)$-inductive $\mathcal{L}^+$-terms and inequalities]\label{Definite:Inductive:Def}
For any $\mathcal{L}^+$-term $s(p_1,\ldots p_n)$  such that the signed generation tree $*s$ $(* \in \{-, + \})$ is  $(\Omega, \epsilon)$-inductive, $*s$ is \emph{definite $(\Omega, \epsilon)$-inductive} if, in addition,
there are no occurences of $+\vee$ or $-\wedge$ nodes in the segment $P_2$ of any $\epsilon$-critical branch of $*s$.
%
The term $s$ is {\em definite} $(\Omega, \epsilon)$-{\em left inductive} (resp.\ {\em definite} $(\Omega, \epsilon)$-{\em right inductive}) if $+s$ (resp.\ $-s$) is definite $(\Omega, \epsilon)$-inductive. An inequality $s \leq t \in \mathcal{L}_{\mathit{term}}$ is \emph{definite $(\Omega, \epsilon)$-inductive} if the trees $+s$ and $-t$ are both definite $(\Omega, \epsilon)$-inductive.
\end{definition}

The definition of definite inductive inequalities is meant to capture the syntactic shape of inductive inequalities after preprocessing (see Lemma \ref{Pre:Process:Lemma} below). During preprocessing, all occurences of $+\lor$ and $-\land$ in the segment $P_2$ of every critical branch can be surfaced and then eliminated via exhaustive applications of the splitting rule.


The proof of the following lemma is analogous to that of \cite[Lemma 10.4]{ConPal12}.
\begin{lemma}\label{Pre:Process:Lemma}
Let $\{ s_i \leq t_i\}$ be the set of inequalities obtained by preprocessing an $(\Omega, \epsilon)$-inductive $\mathcal{L}$-inequality $s \leq t $. Then each $s_i \leq t_i$ is a definite $(\Omega, \epsilon)$-inductive inequality.
\end{lemma}

\begin{definition}[Definite good shape]
An inequality $s \leq t \in \mathcal{L}^+$ is in \emph{definite $(\Omega, \epsilon)$-good shape} if either of the following conditions hold:
\begin{enumerate}
\item $s$ is pure, $+t$ is definite $(\Omega, \epsilon)$-inductive, and moreover, if $+t$ contains a skeleton node on an $\epsilon$-critical branch, then $s$ is a nominal.\footnote{Note that the sides have been swapped around: We require that the \emph{righthand} side of the inequality must be \emph{left} inductive. This is so because the first approximation rule swaps the sides of inequalities.}
\item $t$ is pure, $-s$ is definite $(\Omega, \epsilon)$-inductive and moreover, if $-s$ contains a skeleton node on an $\epsilon$-critical branch, then $t$ is a conominal.
\end{enumerate}
\end{definition}

\noindent Clearly, if an inequality $s\leq t$ is definite $(\Omega, \epsilon)$-inductive, then the two inequalities obtained by applying the first approximation rule to it are in definite $(\Omega, \epsilon)$-good shape. Next, we would like to prove a `good-shape lemma' for definite inductive inequalities. In particular, we would like to show that the application of the reduction rules does not spoil good shape. Actually the application the following rules might spoil good shape:
\begin{prooftree} \AxiomC{$x\wedge y\leq z$}\doubleLine\UnaryInfC{$x\leq y\rightarrow z$}
\AxiomC{$z\leq y \vee x $}\doubleLine\UnaryInfC{$z - y\leq x$}
 \noLine\BinaryInfC{}
\end{prooftree}

This happens e.g.\ when $z$ is pure and $y$ is not. A solution to this is provided by allowing only applications of the rules above which are {\em restricted} to the cases in which the term $y$ that
switches sides is pure.
\begin{lemma}\label{Good:Shape:Lemma}
If $s \leq t$ is in definite $(\Omega, \epsilon)$-good shape, then any inequality $s' \leq t'$ obtained from $s \leq t$, by either the application of a splitting rule, of an approximation rule, or of a residuation rule for a unary connective from-top-to-bottom, or of the application of a residuation rule restricted as indicated above, is again in definite $(\Omega, \epsilon)$-good shape. Moreover, any side condition introduced by an application of an adjunction rule is pure on both sides.
\end{lemma}
\begin{proof}
The proof of the lemma above is analogous to that of \cite[Lemma 10.6]{ConPal12}. We only discuss the rules and the additional statement specific to the regular setting. Consider for instance the case in which $s$ is definite inductive, $t$ is pure and the root of $s$ is $f$ for $\eta_f = 1$. Then the following adjunction rule is applicable:
\begin{center}
\AxiomC{$f(s')\leq t$}
\RightLabel{}
\UnaryInfC{$f(\bot)\leq t\;\;\;s'\leq\blacksquare_{f} t$}
\DisplayProof
\end{center}
Then, both inequalities in the conclusion are in definite $(\Omega, \epsilon)$-good shape. Indeed, the side-condition is all pure, as required by the second part of the statement, and $s'\leq\blacksquare_{f} t$ is in definite good shape, because otherwise, the inequality $f(s')\leq t$ would not be, contrary to the assumptions. The remaining cases are analogous and are omitted.
\end{proof}
\begin{definition}[$(\Omega, \epsilon)$-Ackermann form]\label{Omega:Epsilon:Acrmann:Def}
A set of $\mathcal{L}^+$-inequalities $\{s_i \leq t_i \}_{i \in I}$ is in \emph{reduced $1$-Ackermann form with respect to a variable $p$} if, for every $i \in I$, either

\begin{enumerate}
\item $s_i$ is pure and $t_i = p$, or
\item $s_i$ is positive in $p$ and $t_i$ is negative in $p$.
\end{enumerate}
Similarly, the set $\{s_i \leq t_i \}_{i \in I}$ is in \emph{reduced $\partial$-Ackermann form} with respect to a variable $p$ if, for every $i \in I$, either
\begin{enumerate}
\item $s_i = p$ and $t_i$ is pure, or
\item $s_i$ is negative in $p$ and $t_i$ is positive in $p$.
\end{enumerate}
For any irreflexive, transitive ordering $\Omega$ on $p_1, \ldots, p_n$ and any order-type $\epsilon = (\epsilon_1, \ldots, \epsilon_n)$, a set $\{s_i \leq t_i \}_{i \in I}$ of inequalities is in \emph{reduced $(\Omega, \epsilon)$-Ackermann form} if it is in reduced $\epsilon_k$-Ackermann form with respect to every $\Omega$-minimal variable $p_k$.
\end{definition}

\begin{prop}\label{Trnsfrm:Indctv:To:Rdcd:Acrmn:Lemma}
Any finite set $\{s_i \leq t_i \}_{i \in I}$ of inequalities which are in definite $(\Omega, \epsilon)$-good shape can be transformed into a set $\{s'_i \leq t'_i \}_{i \in I'}$ which is in reduced $(\Omega, \epsilon)$-Ackermann form, through the exhaustive and safe application, only to non-pure inequalities, of the
$\wedge$-splitting, $\vee$-splitting, approximation, and unary residuation rules top-to-bottom, as well as the restricted application of the binary residuation rules top-to-bottom.
\end{prop}
\begin{proof}
The proof is analogous to that of \cite[Proposition 10.10]{ConPal12}, and makes use of the counterparts, in the $\mathcal{L}$-setting, of \cite[Lemma 10.8, Corollary 10.9]{ConPal12}, which are here omitted, since their statement and proof reproduce the mentioned ones verbatim. Notice that by Lemma \ref{Good:Shape:Lemma} the side conditions do not contain proposition variables, which guarantees that if the rules are applied to display critical variable occurrences, then they are applied safely.
\end{proof}
\begin{theorem}\label{Thm:ALBA:Success:Inductive}
For each inductive inequality, there exists a safe and successful execution of \textsf{ALBA}$^r$ on it.
\end{theorem}

\begin{proof}
Let $s_0 \leq t_0$ be an $(\Omega, \epsilon)$-inductive inequality. By Lemma \ref{Pre:Process:Lemma}, preprocessing $s_0\leq t_0$ will yield a finite set $\{s_i \leq t_i \}_{i \in I}$ of definite $(\Omega, \epsilon)$-inductive inequalities. The execution of the algorithm now branches and proceeds separately on each of these inequalities. Each $s \leq t \in \{s_i \leq t_i\}_{i \in I}$ is replaced with $\{ \nomi \leq s, t \leq \cnomm\}$. Notice that $\nomi \leq s$ and $t \leq \cnomm$ are in definite $(\Omega, \epsilon)$-good shape. Hence, by Proposition \ref{Trnsfrm:Indctv:To:Rdcd:Acrmn:Lemma}, the system $\{ \nomi \leq s, t \leq \cnomm\}$ can be transformed, through the \emph{safe} application of the rules of the algorithm, into a set of inequalities in reduced $(\Omega, \epsilon)$-Ackermann form. To this set, the Ackermann-rule can then be applied to eliminate all $\Omega$-minimal propositional variables.

\noindent The Ackermann rule, applied to a set of inequalities in reduced $(\Omega, \epsilon)$-Ackermann form, replaces propositional variables with pure terms, therefore the resulting set of inequalities is in definite $(\Omega', \epsilon)$-good shape, where $\Omega'$ is the restriction of $\Omega$ to the non
$\Omega$-minimal variables. Indeed, the application of an Ackermann rule turns all $\epsilon$-critical branches corresponding to $\Omega$-minimal variables into non-critical branches, and leaves the critical branches corresponding to the other variables unaffected.

Now another cycle of reduction rules, applied safely, will lead to a new set of inequalities in reduced $(\Omega', \epsilon)$-Ackermann form, from which an application of the Ackermann rule will eliminate all the $\Omega'$-minimal variables, and so on. Since the number of variables in $s_0 \leq t_0$ is finite, after a finite number of cycles the algorithm will output a set of pure inequalities.
\end{proof}

As a corollary of Theorems \ref{Crctns:Theorem}, \ref{thm:canonical}, and \ref{Thm:ALBA:Success:Inductive} we obtain:
\begin{theorem}
All inductive $\mathcal{L}$-inequalities are elementary and canonical.
\end{theorem}

\section*{Part III: Applications to Lemmon's logics}
\addcontentsline{toc}{section}{\underline{Part III: Applications to Lemmon's logics}}

\section{Strong completeness and elementarity of E2-E5}\label{sec:example:lemmon:corre:canoni}
In the present section, we apply the theory developed so far to Lemmon's logics E2-E5 
(cf.\ page \pageref{pageref:lemmonsystem}). As to E2-E4, we will show that they are strongly complete with respect to elementary classes of Kripke models with impossible worlds. Moreover, we will give a semantic proof to Kripke's statement in \cite{Kr65} that E5 coincides with S5. We have already seen (cf.\ Example \ref{eg:lemmonsahlqvist}) that each axiom involved in the axiomatization of these logics is Sahlqvist. By Theorem \ref{thm:completeness}, this implies that E2-E5 are strongly complete and elementary. In Subsection \ref{ssec:standard translation}, we adapt the definition of standard translation given in \cite[Subsection 2.5.2]{ConPal12} to the setting of Kripke frames with impossible worlds. In Subsection \ref{ssec:alba on E2 E4}, we effectively compute the first-order conditions defining their associated classes of Kripke frames with impossible worlds by providing a successful and safe run of \textsf{ALBA}$^r$ on each axiom.

\subsection{Standard translation}
\label{ssec:standard translation}
Let $L_1$ be the first-order language with equality with binary relation symbols $R$, and unary predicate symbols $P, Q, \ldots$ corresponding to the propositional variables $p, q, \ldots \in \mathsf{AtProp}$ and unary predicate symbol $N$ for normal worlds. As usual, we let $L_1$ contain a denumerable infinity of individual variables. We will further assume that $L_1$ contains denumerably infinite individual variables $i, j, \ldots$ corresponding to the nominals $\nomi, \nomj, \ldots \in \mathsf{Nom}$ and $n, m, \ldots$ corresponding to the conominals $\cnomn, \cnomm \in \mathsf{CNom}$. Let $L_0$ be the sub-language which does not contain the unary predicate symbols $P, Q, \ldots$ corresponding to the propositional variables.
Let us now define the \emph{standard translation} of ${\mathcal{L}^+_{\eta}}$ into $L_1$ recursively:
\begin{center}
\begin{tabular}{ll}
\begin{tabular}{l c l}
$\mathrm{ST}_x(\bot)$ & $:=$& $x \not\equiv x$ \\

$\mathrm{ST}_x(\top) $ & $:=$& $ x\equiv x$\\
$\mathrm{ST}_x(p) $ & $:=$ &$ P(x)$\\
$\mathrm{ST}_x(\nomj) $ & $:=$& $ j\equiv x$\\
 $\mathrm{ST}_x(\cnomm) $ & $:=$& $ x \not\equiv m$\\
$\mathrm{ST}_x(\neg \varphi) $ & $:=$ &$\neg \mathrm{ST}_x(\varphi)$\\
$\mathrm{ST}_x(\varphi\to\psi) $ & $:=$ &$ \mathrm{ST}_x(\varphi)\to \mathrm{ST}_x(\psi)$\\

\end{tabular}
&
\begin{tabular}{l c l}
$\mathrm{ST}_x(\phi \vee \psi) $ & $:=$& $ \mathrm{ST}_x(\phi) \vee \mathrm{ST}_x(\psi)$\\
$\mathrm{ST}_x(\phi \wedge \psi) $ & $:=$& $ \mathrm{ST}_x(\phi) \wedge \mathrm{ST}_x(\psi)$\\
$\mathrm{ST}_x(\Diamond \phi) $ & $:=$ &$ \neg Nx\lor\exists y (Rxy \wedge \mathrm{ST}_y(\phi))$\\
 $\mathrm{ST}_x(\Box \phi) $ & $:=$ &$ Nx\land\forall y (Rxy \rightarrow \mathrm{ST}_y(\phi))$\\
$\mathrm{ST}_x(\Diamondblack \phi) $ & $:=$& $ \exists y (Ryx \wedge \mathrm{ST}_y(\phi))$\\
$\mathrm{ST}_x({\blacksquare} \phi) $ & $:=$ & $ \forall y (Ryx \rightarrow \mathrm{ST}_y(\phi))$
\end{tabular}
\end{tabular}
\end{center}

$\mathrm{ST}_x$ extends to inequalities and quasi-inequalities as follows: for inequalities, $\mathrm{ST}_x(\phi \leq \psi) := \mathrm{ST}_x(\phi) \rightarrow \mathrm{ST}_x(\psi)$, and for quasi-inequalities, $\mathrm{ST}_x(\phi_1 \leq \psi_1 \amp \cdots \amp \phi_n \leq \psi_n \Rightarrow \phi \leq \psi) := [\mathrm{ST}_x(\phi_1 \leq \psi_1) \wedge \cdots \wedge \mathrm{ST}_x(\phi_n \leq \psi_n)] \rightarrow \mathrm{ST}_x(\phi \leq \psi)$. We also extend $\mathrm{ST}_{x}$ to finite sets of inequalities by declaring $\mathrm{ST}_{x}(\{\phi_1 \leq \psi_1, \ldots, \phi_n \leq \psi_n \}): = \bigwedge_{1 \leq i \leq n} \mathrm{ST}_x(\phi_i \leq \psi_i)$.

Observe that if $\mathrm{ST}_x$ is applied to pure terms, inequalities, or quasi-inequalities (see page \pageref{def:language}), it produces formulas in the sublanguage $L_0$. The following lemma is proved by a routine induction.

\begin{lemma}\label{lemma:Standard:Translation}
For any state $w$ in a Kripke frame with impossible worlds $\mathcal{F}$ and for every formula, inequality or quasi-inequality $\xi$ in the language $\mathcal{L}^+_{\eta}$,
\begin{enumerate}
\item $\mathcal{F}, w \Vdash \xi\quad $ iff $\quad \mathcal{F} \models \forall \overline{P} \forall \overline{j} \forall \overline{m} \mathrm{ST}_x(\xi) [x:= w]$, and
\item $\mathcal{F} \Vdash \xi\quad $ iff $\quad\mathcal{F} \models \forall x \forall \overline{P} \forall \overline{j} \forall \overline{m} \mathrm{ST}_x(\xi)$,
\end{enumerate}
where $\overline{P}$, $\overline{j}$, and $\overline{m}$ are, respectively, the vectors of all predicate symbols corresponding to propositional variables, individual variables corresponding to nominals, and individual variables corresponding to conominals, occurring in $\mathrm{ST}_x(\xi)$.
\end{lemma}

\subsection{Strong completeness and elementarity of E2-E5}
\label{ssec:alba on E2 E4}
In the present subsection, we provide successful and safe runs of the algorithm $\mathsf{ALBA}^{r}$ on each of the axioms in E2-E5. Thereafter, we use the standard translation defined in the previous subsection to give an interpretation of the pure inequalities on Kripke frames with impossible worlds.
\begin{prop}\label{thm:lemmonalgcan}
The following axioms in Lemmon's system are canonical.

\begin{center}
\vspace{2mm}
\begin{tabular}{llll}
\vspace{2mm}
(1) & $\Box(p\to q)\to\Box(\Box p\to\Box q)$ & (1') & $\Box(p\to q)\to(\Box p\to\Box q)$\\
\vspace{2mm}
(2) & $\Box p\to p$ & & \\ 
\vspace{2mm}
(4) & $\Box p\to\Box\Box p$ & (5) & $\neg\Box p\to\Box\neg\Box p$.\\
\end{tabular}
\end{center}
\end{prop}

\begin{proof}
As discussed in Example \ref{eg:lemmonsahlqvist}, the above axioms are DLR-Sahlqvist. Hence, by Theorem \ref{thm:additivecan}, they are canonical.
\end{proof}
In what follows, we compute the first-order frame correspondents of the above axioms. By Theorem \ref{Crctns:Theorem}, each of the logics E2-E5 is strongly complete with respect to the class of frames defined by the first-order frame correspondents of its axioms (see Tables \ref{table:elementary}, \ref{table:modal:elementary} and \ref{table:lemmonsystem}).

\paragraph{}$(2)\;\; \Box p\rightarrow p$
\begin{center}
\begin{tabular}{ll}

$\forall p(\Box p\leq p)$&\\

$\forall p\forall\nomi\forall\cnomm [(\nomi\leq \Box p\;\&\; p\leq \cnomm)\Rightarrow \nomi\leq\cnomm]$&\\

$\forall p\forall\nomi\forall\cnomm [(\nomi\leq \Box \top \;\&\; \Diamondblack \nomi\leq p \;\&\; p\leq \cnomm)\Rightarrow \nomi\leq\cnomm]$& (adjunction rule for $\Box$)\\

$\forall\nomi\forall\cnomm [(\nomi\leq \Box \top \;\&\;\Diamondblack \nomi\leq \cnomm)\Rightarrow \nomi\leq\cnomm]$ & (Ackermann rule)\\

$\forall\nomi [\nomi\leq \Box \top \Rightarrow \nomi\leq\Diamondblack \nomi]$.\\

\end{tabular}
\end{center}
Using the standard translation, the last clause above translates to the sentence below, which is then further simplified:

\begin{center}
\begin{tabular}{ll}

$\forall i \forall x ((i\equiv x \to Nx)\to (i\equiv x\to \exists y(Ryx\land i\equiv y)))$\\
$\forall i\forall x (Ni\to \exists y (Ryi\land i\equiv y))$\\
$\forall i(Ni\to Rii)$.\\

\end{tabular}
\end{center}

\paragraph{}$(4)\;\; \Box p\rightarrow \Box\Box p$
\begin{center}
\begin{tabular}{ll}

$\forall p(\Box p\leq\Box\Box p)$&\\

$\forall p\forall\nomi\forall\cnomm [(\nomi\leq \Box p\;\&\; \Box\Box p\leq \cnomm)\Rightarrow \nomi\leq\cnomm]$&\\

$\forall p\forall\nomi\forall\cnomm [(\nomi\leq \Box \top \;\&\; \Diamondblack \nomi\leq p \;\&\;\Box\Box p\leq \cnomm)\Rightarrow \nomi\leq\cnomm]$& (adjunction rule for $\Box$)\\

$\forall\nomi\forall\cnomm [(\nomi\leq \Box \top \;\&\;\Box\Box \Diamondblack \nomi\leq \cnomm)\Rightarrow \nomi\leq\cnomm]$ & (Ackermann rule)\\

$\forall\nomi [\nomi\leq \Box \top \Rightarrow \nomi\leq\Box\Box\Diamondblack \nomi]$.\\

\end{tabular}
\end{center}
Using the standard translation and after some simplifying steps, the clause above we get
\begin{center}
$\forall i (Ni\to\forall y (Riy\to Ny\land \forall z(Ryz\to Riz)))$,
\end{center}

which is equivalent to
\begin{center}
$\forall i\forall y\forall z ( Ni\land Ny\land Riy\land Ryz\to Riz)\land \forall i\forall y( Ni\land Riy\to Ny).$
\end{center}

\paragraph{}The validity of $(5)\;\neg\Box p\to\Box\neg\Box p$ is equivalent to the validity of $\Diamond p\to\Box\Diamond p$.
\begin{center}
\begin{tabular}{ll}

\vspace{2mm}

$\forall p (\Diamond p\to\Box\Diamond p)$&\\

\vspace{2mm}

$\forall p\forall\nomi\forall\cnomm [(\nomi\leq \Diamond p\;\&\; \Box\Diamond p\leq \cnomm)\Rightarrow \nomi\leq\cnomm]$&\\

\vspace{2mm}

$\forall p\forall\nomi\forall\nomj\forall\cnomm [((\nomi\leq \Diamond \bot\;\parr\;(\nomi\leq \Diamond \nomj\;\&\;\nomj\leq p))\;\&\; \Box\Diamond p\leq \cnomm)\Rightarrow \nomi\leq\cnomm]$& (approximation rule for $\Diamond$)\\

$\forall p\forall\nomi\forall\cnomm [(\nomi\leq \Diamond \bot\;\&\; \Box\Diamond p\leq \cnomm)\Rightarrow \nomi\leq\cnomm]\;\&$&\\

\vspace{2mm}

$\;\forall p\forall\nomi\forall \nomj\forall\cnomm [(\nomi\leq \Diamond \nomj\;\&\;\nomj\leq p\;\&\; \Box\Diamond p\leq \cnomm)\Rightarrow \nomi\leq\cnomm]$&\\

$\forall\nomi\forall\cnomm [(\nomi\leq \Diamond \bot\;\&\; \Box\Diamond \bot\leq \cnomm)\Rightarrow \nomi\leq\cnomm]\;\&$ & \\

\vspace{2mm}

$\;\forall\nomi\forall \nomj\forall\cnomm [(\nomi\leq \Diamond \nomj\;\&\;\Box\Diamond \nomj\leq \cnomm)\Rightarrow \nomi\leq\cnomm]$ & (Ackermann rule)\\

\vspace{2mm}

$[\Diamond \bot\leq\Box\Diamond \bot]\;\&\;\forall \nomj[(\Diamond \nomj\leq\Box\Diamond \nomj]$.&\\

\end{tabular}
\end{center}

Using the standard translation and after some simplifying steps, the clause above can be rewritten as:
\begin{center}
$\forall xNx \land \forall x \forall y\forall z(Nx\land Ny\land Rxy\land Rxz\to Ryz).$
\end{center}

In what follows, we run the algorithm $\mathsf{ALBA}^{r}$ with respect to the natural order-type $\epsilon_p=1, \epsilon_q=1$. Notice that the axioms (1) and (1') in Lemmon's system are inductive but not DLR-Sahlqvist with respect to this order-type.
Since the execution of $\mathsf{ALBA}^{r}$ is successful on these axioms, using Theorem \ref{thm:canonical}, they are canonical.
\bigskip

\noindent $(1)\;\; \Box(p\to q)\to\Box(\Box p\to\Box q)$
\begin{itemize}
\item[\phantom{(1)}]
$\forall p\forall q\forall\nomi\forall\cnomm[(\nomi\leq \Box (p\rightarrow q)\; \& \;\Box(\Box p\rightarrow\Box q)\leq \cnomm)\Rightarrow \nomi\leq \cnomm]$

\vspace{2mm}
$\forall p\forall q\forall\nomi\forall\cnomm[(\nomi\leq\Box\top\;\&\;\Diamondblack\nomi\leq p\rightarrow q\;\&\; \Box(\Box p\rightarrow
\Box q)\leq \cnomm)\Rightarrow \nomi\leq \cnomm]$
\vspace{2mm}

$\forall p\forall q\forall\nomi\forall\cnomm[(\nomi\leq\Box\top\;\&\;\Diamondblack\nomi\land p\leq q\;\&\; \Box(\Box p\rightarrow
\Box q)\leq \cnomm)\Rightarrow \nomi\leq \cnomm]$

\vspace{2mm}
$\forall p\forall q\forall\nomi\forall\cnomm\forall\cnomn[(\nomi\leq\Box\top\;\&\;\Diamondblack\nomi\land p\leq q\;\&\;((\Box\top\leq \cnomm) \;\parr\; (\Box\cnomn\leq \cnomm\; \&\; \Box p \rightarrow
\Box q\leq \cnomn)))\Rightarrow \nomi\leq \cnomm]$

\vspace{2mm}
$\forall p\forall q\forall\nomi\forall\cnomm[(\nomi\leq\Box\top\;\&\;\Diamondblack\nomi\land p\leq q\;\&\;\Box\top\leq \cnomm)\Rightarrow \nomi\leq \cnomm]\;\&$

$\;\forall p\forall q\forall\nomi\forall\cnomm\forall\cnomn[(\nomi\leq\Box\top\;\&\;\Diamondblack\nomi\land p\leq q\;\&\;\Box\cnomn\leq \cnomm\; \&\; \Box p \rightarrow
\Box q\leq \cnomn)\Rightarrow \nomi\leq \cnomm]$.
\end{itemize}
Notice that the first of the two quasi-inequalities above is a tautology: indeed, $\nomi\leq\Box\top$ and $\Box\top\leq \cnomm$ imply $\nomi\leq \cnomm$. Hence, the clause above simplifies to:

\begin{itemize}
\item[\phantom{(1)}] $\forall p\forall q\forall\nomi\forall\cnomm\forall\cnomn[(\nomi\leq\Box\top\;\&\;\Diamondblack\nomi\land p\leq q\;\&\;\Box\cnomn\leq \cnomm\; \&\; \Box p \rightarrow\Box q\leq \cnomn)\Rightarrow \nomi\leq \cnomm]$

\vspace{2mm}

$\forall p\forall q\forall\nomi\forall \nomi_0\forall\cnomm\forall\cnomn\forall\cnomn_0[(\nomi\leq\Box\top\;\&\;\Diamondblack\nomi\land p\leq q\;\&\;\Box\cnomn\leq \cnomm\; \&$

$\;\;\;\;\;\;\;\;\;\;\;\;\;\;\;\;\;\;\;\;\;\;\;\;\;\;\;\;\;\;\;\;\;\nomi_0\leq\Box p\;\&\;\Box q\leq\cnomn_0\;\&\;\nomi_0\rightarrow\cnomn_0\leq\cnomn)\Rightarrow \nomi\leq \cnomm]$

\vspace{2mm}

$\forall p\forall q\forall\nomi\forall \nomi_0\forall\cnomm\forall\cnomn\forall\cnomn_0[(\nomi\leq\Box\top\;\&\;\Diamondblack\nomi\land p\leq q\;\&\;\Box\cnomn\leq \cnomm\; \&$

$\;\;\;\;\;\;\;\;\;\;\;\;\;\;\;\;\;\;\;\;\;\;\;\;\;\;\;\;\;\;\;\;\;\nomi_0\leq\Box\top\; \&\;\Diamondblack\nomi_0\leq p\;\&\;\Box q\leq\cnomn_0\;\&\;\nomi_0\rightarrow\cnomn_0)\leq\cnomn\Rightarrow \nomi\leq \cnomm]$.
\end{itemize}
By applying the right-hand Ackermann rule to $p$, we have:

\begin{itemize}
\item[\phantom{(1)}] $\forall q\forall\nomi\forall \nomi_0\forall\cnomm\forall\cnomn\forall\cnomn_0[(\nomi\leq\Box\top\;\&\;\Diamondblack\nomi\land \Diamondblack\nomi_0\leq q\;\&\;\Box\cnomn\leq \cnomm\; \&\;\nomi_0\leq\Box\top\; \&\;\Box q\leq\cnomn_0\;\&$

$\;\;\;\;\;\;\;\;\;\;\;\;\;\;\;\;\;\;\;\;\;\;\;\;\;\;\;\;\;\;\;\nomi_0\rightarrow\cnomn_0\leq\cnomn)\Rightarrow \nomi\leq \cnomm]$.
\end{itemize}
By applying the right-hand Ackermann rule to $q$, we have:

\begin{itemize}
\item[\phantom{(1)}] $\forall\nomi\forall \nomi_0\forall\cnomm\forall\cnomn\forall\cnomn_0[(\nomi\leq\Box\top\;\&\;\Box\cnomn\leq \cnomm\; \&\;\nomi_0\leq\Box\top\; \&\;\Box(\Diamondblack\nomi\land \Diamondblack\nomi_0)\leq\cnomn_0\; \&\;\nomi_0\rightarrow\cnomn_0\leq\cnomn)\Rightarrow \nomi\leq \cnomm]$,
\end{itemize}
which simplifies to:

\begin{itemize}
\item[\phantom{(1)}]$\forall\nomi\forall \nomi_0[(\nomi\leq\Box\top\; \&\;\nomi_0\leq\Box\top)\; \Rightarrow(\Diamondblack\nomi\land\nomi_0)\leq\Box(\Diamondblack\nomi\land \Diamondblack\nomi_0)]$.
\end{itemize}
Using the standard translation, it simplifies to:

\begin{itemize}
\item[\phantom{(1)}]$\forall i\forall i_0\forall y (N i\land Ni_0\land R i i_0\land Ri_0 y\to R iy) $.
\end{itemize}
\bigskip

Let us show that the axiom (1') is valid on each Kripke frame with impossible worlds.

\noindent $(1')\;\; \Box(p\to q)\to(\Box p\to\Box q)$
\begin{itemize}
\item[\phantom{(1)}]
$\forall p\forall q\forall\nomi\forall\cnomm[(\nomi\leq \Box (p\rightarrow q)\; \& \;\Box p\rightarrow
\Box q\leq \cnomm)\Rightarrow \nomi\leq \cnomm]$\\
$\forall p\forall q\forall\nomi\forall\cnomm[(\nomi\leq\Box\top\;\&\;\Diamondblack\nomi\leq p\rightarrow q\;\&\; \Box p\rightarrow
\Box q\leq \cnomm)\Rightarrow \nomi\leq \cnomm]$\\
$\forall p\forall q\forall\nomi\forall\cnomm[(\nomi\leq\Box\top\;\&\;\Diamondblack\nomi\land p\leq q\;\&\; \Box p\rightarrow
\Box q\leq \cnomm)\Rightarrow \nomi\leq \cnomm]$\\
$\forall p\forall q\forall\nomi\forall\nomj\forall\cnomm\forall\cnomn[(\nomi\leq\Box\top\;\&\;\Diamondblack\nomi\land p\leq q\;\&\; \Box q\leq\cnomn\;\&\; \nomj\leq\Box p\;\&\; \nomj\rightarrow\cnomn\leq \cnomm)\Rightarrow \nomi\leq \cnomm]$\\
$\forall p\forall q\forall\nomi\forall\nomj\forall\cnomm\forall\cnomn[(\nomi\leq\Box\top\;\&\;\Diamondblack\nomi\land p\leq q\;\&\; \Box q\leq\cnomn\;\&\;\Diamondblack\nomj\leq p\;\&\; \nomj\leq\Box\top\;\&\; \nomj\rightarrow\cnomn\leq \cnomm)\Rightarrow \nomi\leq \cnomm]$.
\end{itemize}
By applying the right-hand Ackermann rule to $p$, we have:
\begin{itemize}
\item[\phantom{(1)}]
$\forall q\forall\nomi\forall\nomj\forall\cnomm\forall\cnomn[(\nomi\leq\Box\top\;\&\;\Diamondblack\nomi\land \Diamondblack\nomj\leq q\;\&\; \Box q\leq\cnomn\;\&\; \nomj\leq\Box\top\;\&\; \nomj\rightarrow\cnomn\leq \cnomm)\Rightarrow \nomi\leq \cnomm]$.
\end{itemize}
By applying the right-hand Ackermann rule to $q$, we have:
\begin{itemize}
\item[\phantom{(1)}]
$\forall\nomi\forall\nomj\forall\cnomm\forall\cnomn[(\nomi\leq\Box\top\;\&\;\Box (\Diamondblack\nomi\land \Diamondblack\nomj)\leq\cnomn\;\&\; \nomj\leq\Box\top\;\&\; \nomj\rightarrow\cnomn\leq \cnomm)\Rightarrow \nomi\leq \cnomm]$.
\end{itemize}
In order to show that the clause above is a tautology, let us further simplify it by rewriting it as follows:
\begin{itemize}
\item[\phantom{(1)}]
$\forall\nomi\forall\nomj\forall\cnomn[(\nomi\leq\Box\top\;\&\;\Box (\Diamondblack\nomi\land \Diamondblack\nomj)\leq\cnomn\;\&\; \nomj\leq\Box\top)\Rightarrow \nomi\leq\nomj\rightarrow\cnomn]$\\
$\forall\nomi\forall\nomj\forall\cnomn[(\nomi\leq\Box\top\;\&\;\Box (\Diamondblack\nomi\land \Diamondblack\nomj)\leq\cnomn\;\&\; \nomj\leq\Box\top)\Rightarrow \nomi\land\nomj\leq\cnomn]$\\
$\forall\nomi\forall\nomj[(\nomi\leq\Box\top\;\&\;\nomj\leq\Box\top)\Rightarrow \nomi\land\nomj\leq\Box (\Diamondblack\nomi\land \Diamondblack\nomj)]$\\
$\forall\nomi\forall\nomj[(\nomi\leq\Box\top\;\&\;\nomj\leq\Box\top)\Rightarrow \nomi\land\nomj\leq\Box\Diamondblack\nomi\land\Box\Diamondblack\nomj]$\\
$\forall\nomi\forall\nomj[(\nomi\leq\Box\top\;\&\;\nomj\leq\Box\top)\Rightarrow (\nomi\land\nomj\leq\Box\Diamondblack\nomi\;\;\&\;\;\nomi\land\nomj\leq\Box\Diamondblack\nomj)]$\\
$\forall\nomi\forall\nomj[(\nomi\leq\Box\top\;\&\;\nomj\leq\Box\top)\Rightarrow \nomi\land\nomj\leq\Box\Diamondblack\nomi]\;\&\;\forall\nomi\forall\nomj[(\nomi\leq\Box\top\;\&\;\nomj\leq\Box\top)\Rightarrow \nomi\land\nomj\leq\Box\Diamondblack\nomj].$
\end{itemize}
The two quasi-inequalities above are respectively implied by
\begin{center}
$\forall\nomi(\nomi\leq\Box\top\Rightarrow\nomi\leq\Box\Diamondblack\nomi)\mbox{\ \ and\ \ }\forall\nomj(\nomj\leq\Box\top\Rightarrow\nomj\leq\Box\Diamondblack\nomj)$.
\end{center}
By applying the adjunction rule to the inequality in the consequent, we get:
\begin{center}
$\forall\nomi(\nomi\leq\Box\top\Rightarrow\nomi\leq\Box\top\;\&\;\Diamondblack\nomi\leq\Diamondblack\nomi)$,
\end{center}
which is implied by the tautology
\begin{center}
$\forall\nomi(\Diamondblack\nomi\leq\Diamondblack\nomi)$.
\end{center}
The correspondence results obtained above are summarized in the following tables:
\begin{center}
\begin{table}[H]
\begin{tabular}{|c|l|}
\hline
\textbf{Elementary frame condition} & \textbf{First-order formula}\\
\hline
Normality & $\forall x Nx$\\
\hline
Closure under normality & $\forall x\forall y (Nx\land Rxy\to Ny)$\\
\hline
Pre-normal reflexivity& $\forall x (Nx\to Rxx)$\\
\hline
Pre-normal transitivity& $\forall x \forall y\forall z(Nx\land Ny\land Rxy\land Ryz\to Rxz)$\\
\hline
Pre-normal euclideanness & $\forall x \forall y\forall z(Nx\land Ny\land Rxy\land Rxz\to Ryz)$\\
\hline
\end{tabular}
\caption{Elementary frame conditions}
\label{table:elementary}
\end{table}
\end{center}

\begin{center}
\begin{table}[H]
\begin{tabular}{|c|l|}
\hline
\textbf{Modal axiom} & \textbf{Elementary frame condition}\\
\hline
$\Box p\rightarrow p$ & Pre-normal reflexivity\\
\hline
$\Box p\rightarrow \Box\Box p$ & Pre-normal transitivity and closure under normality\\
\hline
$\neg\Box p\rightarrow \Box\neg\Box p$ & Normality and pre-normal euclideanness\\
\hline
$\Box(p\rightarrow q)\rightarrow (\Box p\rightarrow \Box q)$ & $\top$\\
\hline
$\Box(p\rightarrow q)\rightarrow \Box(\Box p\rightarrow \Box q)$ & Pre-normal transitivity\\
\hline
\end{tabular}
\caption{Lemmon's modal axioms and their elementary frame conditions}
\label{table:modal:elementary}
\end{table}
\end{center}

\begin{theorem}\label{thm:completeness}
Each of E2, E3, E4, E5 is strongly complete with respect to the class of frames specified in Table \ref{table:lemmonsystem}.
\end{theorem}
\begin{proof}
The canonicity of axioms (1), (1'), (2), (4) and (5) is shown in Proposition \ref{thm:lemmonalgcan}.
As is shown above, $\mathsf{ALBA}^{r}$ (safely) succeeds on each of them. Hence, by Theorem \ref{Crctns:Theorem}, the statement follows.
\end{proof}
\begin{center}
\begin{table}[H]
\begin{tabular}{|c|l|}
\hline
\textbf{Lemmon's system} & \textbf{Elementary class of frames}\\
\hline
E2 & Pre-normal reflexivity\\
\hline
E3 & Pre-normal reflexivity and pre-normal transitivity \\
\hline
E4 & Pre-normal reflexivity, pre-normal transitivity and closure under normality\\
\hline
E5 & Pre-normal reflexivity, pre-normal euclideanness and normality\\
\hline
\end{tabular}
\caption{Lemmon's systems and their elementary classes of frames}
\label{table:lemmonsystem}
\end{table}
\end{center}
Notice that any Kripke frame with impossible worlds which satisfies pre-normal reflexivity, pre-normal euclideanness and normality can be uniquely associated with a standard Kripke frame, the binary relation of which is an equivalence relation. Conversely, any such standard Kripke frame can be uniquely associated with a Kripke frame with impossible worlds which satisfies pre-normal reflexivity, pre-normal euclideanness and normality. These observations provide a semantic proof of Kripke's statement that E5 coincides with S5.

\bibliographystyle{abbrv}
\bibliography{aiml14}

\begin{thebibliography}{10}

\bibitem{Barwise1997}
J.~Barwise.
\newblock Information and impossibilities.
\newblock {\em Notre Dame Journal of Formal Logic}, 38(4):488--515, 10 1997.

\bibitem{Be13}
F.~Berto.
\newblock Impossible worlds.
\newblock In E.~N. Zalta, editor, {\em The Stanford Encyclopedia of
  Philosophy}. Winter 2013 edition, 2013.

\bibitem{CH90}
B.~Chellas.
\newblock {\em Modal Logic: An Introduction}.
\newblock Cambridge University Press, 1980.

\bibitem{CoCr14}
W.~Conradie and A.~Craig.
\newblock {Canonicity results for mu-calculi: an algorithmic approach}.
\newblock {\em Journal of Logic and Computation}, forthcoming.

\bibitem{CoFoPaSo}
W.~Conradie, Y.~Fomatati, A.~Palmigiano, and S.~Sourabh.
\newblock Algorithmic correspondence for intuitionistic modal mu-calculus.
\newblock {\em Theoretical Computer Science}, 564:30--62, 2015.

\bibitem{CoGhPa14}
W.~Conradie, S.~Ghilardi, and A.~Palmigiano.
\newblock Unified correspondence.
\newblock In A.~Baltag and S.~Smets, editors, {\em Johan van Benthem on Logic
  and Information Dynamics}, volume~5 of {\em Outstanding Contributions to
  Logic}, pages 933--975. Springer International Publishing, 2014.

\bibitem{CoGoVa06}
W.~Conradie, V.~Goranko, and D.~Vakarelov.
\newblock Algorithmic correspondence and completeness in modal logic {I}. {T}he
  core algorithm {SQEMA}.
\newblock {\em Logical Methods in Computer Science}, 2006.

\bibitem{CPZ:constructive}
W.~Conradie and A.~Palmigiano.
\newblock Constructive canonicity of inductive terms.
\newblock Submitted.

\bibitem{ConPal12}
W.~Conradie and A.~Palmigiano.
\newblock Algorithmic correspondence and canonicity for distributive modal
  logic.
\newblock {\em Annals of Pure and Applied Logic}, 163(3):338 -- 376, 2012.

\bibitem{ConPal13}
W.~Conradie and A.~Palmigiano.
\newblock Algorithmic correspondence and canonicity for non-distributive
  logics.
\newblock {\em Journal of Logic and Computation}, forthcoming.

\bibitem{ConPalSou12}
W.~Conradie, A.~Palmigiano, and S.~Sourabh.
\newblock Algebraic modal correspondence: {S}ahlqvist and beyond.
\newblock Submitted.

\bibitem{CGPSZ14}
W.~Conradie, A.~Palmigiano, S.~Sourabh, and Z.~Zhao.
\newblock Canonicity and relativized canonicity via pseudo-correspondence: an
  application of {ALBA}.
\newblock Submitted.

\bibitem{CPZ:Trans}
W.~Conradie, A.~Palmigiano, and Z.~Zhao.
\newblock Sahlqvist via translation.
\newblock Submitted.

\bibitem{ConRob}
W.~Conradie and C.~Robinson.
\newblock On {S}ahlqvist theory for hybrid logic.
\newblock {\em Journal of Logic and Computation}, 2015.

\bibitem{DaPr90}
B.~Davey and H.~Priestley.
\newblock {\em Introduction to Lattices and Order}.
\newblock Cambridge University Press, 2002.

\bibitem{FrPaSa14}
S.~Frittella, A.~Palmigiano, and L.~Santocanale.
\newblock Dual characterizations for finite lattices via correspondence theory
  for monotone modal logic.
\newblock {\em Journal of Logic and Computation}, forthcoming.

\bibitem{GeJo94}
M.~Gehrke and B.~J\'onsson.
\newblock Monotone bounded distributive lattice expansions.
\newblock {\em Mathematica Japonica}, 52(2):197--213, 2000.

\bibitem{GeJo04}
M.~Gehrke and B.~J\'onsson.
\newblock Bounded distributive lattice expansions.
\newblock {\em Mathematica Scandinavica}, 94(94):13--45, 2004.

\bibitem{GeNaVe05}
M.~Gehrke, H.~Nagahashi, and Y.~Venema.
\newblock A {S}ahlqvist theorem for distributive modal logic.
\newblock {\em Annals of Pure and Applied Logic}, 131(1-3):65--102, 2005.

\bibitem{GhMe97}
S.~Ghilardi and G.~Meloni.
\newblock Constructive canonicity in non-classical logics.
\newblock {\em Annals of Pure and Applied Logic}, 86(1):1--32, 1997.

\bibitem{GorankoV06}
V.~Goranko and D.~Vakarelov.
\newblock Elementary canonical formulae: extending {S}ahlqvist's theorem.
\newblock {\em Annals of Pure and Applied Logic}, 141(1-2):180--217, 2006.

\bibitem{GMPTZ}
G.~Greco, M.~Ma, A.~Palmigiano, A.~Tzimoulis, and Z.~Zhao.
\newblock Unified correspondence as a proof-theoretic tool.
\newblock {\em Journal of Logic and Computation}, forthcoming.

\bibitem{HH03}
H.~H. Hansen.
\newblock Monotonic modal logics.
\newblock {\em Master's Thesis, ILLC, University of Amsterdam}, 2003.

\bibitem{Jonsson94}
B.~J\'onsson.
\newblock On the canonicity of {S}ahlqvist identities.
\newblock {\em Studia Logica}, 53:473--491, 1994.

\bibitem{JoTa51}
B.~J{\'o}nsson and A.~Tarski.
\newblock {Boolean algebras with operators. Part I}.
\newblock {\em American Journal of Mathematics}, pages 891--939, 1951.

\bibitem{Kr65}
S.~A. Kripke.
\newblock Semantical analysis of modal logic {II}. {N}on-normal modal
  propositional calculi.
\newblock In J.~W. Addison, A.~Tarski, and L.~Henkin, editors, {\em The Theory
  of Models}. North Holland, 1965.

\bibitem{Le57}
E.~J. Lemmon.
\newblock New foundations for {L}ewis modal systems.
\newblock {\em The Journal of Symbolic Logic}, 22(2):176--186, 1957.

\bibitem{MaZhao15}
M.~Ma and Z.~Zhao.
\newblock Unified correspondence and proof theory for strict implication.
\newblock {\em Journal of Logic and Computation}, forthcoming.

\bibitem{Nolan2013}
D.~Nolan.
\newblock Impossible worlds.
\newblock {\em Philosophy Compass}, 8(4):360--372, 2013.

\bibitem{PaSoZh14}
A.~Palmigiano, S.~Sourabh, and Z.~Zhao.
\newblock J\'onsson-style canonicity for {ALBA}-inequalities.
\newblock {\em Journal of Logic and Computation}, 2015.

\bibitem{Ri52}
H.~Ribeiro.
\newblock A remark on {B}oolean algebras with operators.
\newblock {\em American Journal of Mathematics}, 74(1):163--167, 1952.

\bibitem{Sa75}
H.~Sahlqvist.
\newblock Completeness and correspondence in the first and second order
  semantics for modal logic.
\newblock In S.~Kanger, editor, {\em Studies in Logic and the Foundations of
  Mathematics}, volume~82, pages 110--143. North-Holland, Amsterdam, 1975.

\bibitem{SaVa89}
G.~Sambin and V.~Vaccaro.
\newblock A new proof of {S}ahlqvist's theorem on modal definability and
  completeness.
\newblock {\em Journal of Symbolic Logic}, 54(3):992--999, 1989.

\bibitem{SG71}
K.~Segerberg.
\newblock {\em An essay in classical modal logic. 1 (1971)}.
\newblock Filosofiska studier. Filosofiska f{\"o}reningen och Filosofiska
  institutionen vid Uppsala universitet, 1971.

\bibitem{Seki03}
T.~Seki.
\newblock A {S}ahlqvist theorem for relevant modal logics.
\newblock {\em Studia Logica}, 73(3):383--411, 2003.

\bibitem{Suzuki11}
T.~Suzuki.
\newblock Canonicity results of substructural and lattice-based logics.
\newblock {\em The Review of Symbolic Logic}, 4(01):1--42, 2011.

\bibitem{van1983modal}
J.~van Benthem.
\newblock {\em Modal Logic and Classical Logic}.
\newblock Indices. Monographs in Philosophical Logic and Formal Linguistics.
  Bibliopolis, 1983.

\bibitem{Benthem05}
J.~van Benthem.
\newblock Minimal predicates, fixed-points, and definability.
\newblock {\em Journal of Symbolic Logic}, 70(3):696--712, 2005.

\bibitem{Benthem06}
J.~van Benthem.
\newblock Modal frame correspondences and fixed-points.
\newblock {\em Studia Logica}, 83(1-3):133--155, 2006.

\bibitem{BenthemBH12}
J.~van Benthem, N.~Bezhanishvili, and I.~Hodkinson.
\newblock Sahlqvist correspondence for modal mu-calculus.
\newblock {\em Studia Logica}, 100(1-2):31--60, 2012.

\bibitem{HBMoL}
Y.~Venema.
\newblock Algebras and coalgebras.
\newblock {\em Handbook of modal logic}, 3:331--426, 2006.

\end{thebibliography}
\end{document}